\documentclass[12pt,twoside,leqno]{article}

\usepackage[a4paper,total={13.5cm,23cm}]{geometry}
\usepackage[T1]{fontenc}
\usepackage{ae}
\usepackage{eucal}
\usepackage[utf8]{inputenc}
\usepackage{amsmath,amsfonts,amssymb,amsthm}
\usepackage[hyperfootnotes=false]{hyperref}

\newtheorem{theorem}{Theorem}[section]
\newtheorem{proposition}[theorem]{Proposition}
\newtheorem{definition}[theorem]{Definition}
\newtheorem{lemma}[theorem]{Lemma}
\newtheorem{corollary}[theorem]{Corollary}

\numberwithin{equation}{section}

\renewcommand{\geq}{\geqslant}
\renewcommand{\leq}{\leqslant}

\newcommand{\legendre}[2]{\genfrac {(}{)}{1pt}{}{#1}{#2}}
\newcommand{\w}{\mathfrak w}
\newcommand{\p}{\mathfrak p}
\newcommand{\af}{\mathfrak a}
\newcommand{\m}{\mathfrak m}
\newcommand{\F}{\mathcal F}
\newcommand{\OO}{\mathcal O}
\newcommand{\OK}{\mathcal{O}_{\K}}
\newcommand{\HH}{\mathbb H}
\newcommand{\N}{\mathbb N}
\newcommand{\Z}{\mathbb Z}
\newcommand{\R}{\mathbb R}
\newcommand{\Q}{\mathbb Q}
\newcommand{\C}{\mathbb C}
\newcommand{\K}{K}

\newcommand{\Sl}{\operatorname {Sl}}
\newcommand{\Gl}{\operatorname {Gl}}
\newcommand{\Norm}{\operatorname {N}}

\newcommand{\wf}{\mathfrak{f}}
\newcommand{\wfone}{\mathfrak{f}_1}
\newcommand{\wftwo}{\mathfrak{f}_2}
\newcommand {\PhiNc}{\Phi_N^{\mathrm c}}

\newenvironment {alphenumerate}
   {\renewcommand {\labelenumi}{(\alph {enumi})}
    
    \begin {enumerate}}
   {\end {enumerate}}
\newenvironment {romanenumerate}
{
    \begin {enumerate}}
{\end {enumerate}}
\newenvironment {alpharabenumerate}[1][]
   {\renewcommand {\labelenumi}{(#1\arabic{enumi})}
    \begin {enumerate}}
   {\end {enumerate}}

\begin{document}

\title{Generalised Weber Functions}

\author
{
\parbox [t]{6.3cm}{
Andreas Enge \\
INRIA, LFANT \\
CNRS, IMB, UMR 5251 \\
Univ. Bordeaux, IMB \\
33400 Talence \\
France \\
andreas.enge@inria.fr
}
\parbox [t]{6.3cm}{
Fran\c{c}ois Morain \\
INRIA Saclay--\^Ile-de-France \\
\& LIX (CNRS/UMR 7161) \\
\'Ecole polytechnique \\
91128 Palaiseau Cedex \\
France \\
morain@lix.polytechnique.fr
}
}
\date{20 December 2013}

\maketitle

\begin{abstract}
A generalised Weber function is given by $\w_N(z) = \eta(z/N)/\eta(z)$,
where $\eta(z)$ is the Dedekind function and $N$ is any integer; the
original function corresponds to $N=2$. We classify
the cases where some power $\w_N^e$ evaluated at
some quadratic integer
generates the ring class field associated to an order of an imaginary
quadratic field. We
compare the heights of our invariants by giving a general formula for
the degree of the modular equation relating $\w_N(z)$ and
$j(z)$. Our ultimate goal is the use of these invariants in
constructing reductions of elliptic curves over finite fields suitable
for cryptographic use.
\end{abstract}

\renewcommand{\thefootnote}{}

\section{Introduction}

Let $\K$ be an imaginary quadratic field of discriminant $\Delta <
0$. We are interested in orders $\OO$ of $\K$ having
discriminant $D = c^2 \Delta$. The principal order of discriminant
$\Delta$ is $\OK$, which is
generated by $\omega = \frac {1 + \sqrt \Delta}{2}$ if
$\Delta \equiv 1 \pmod 4$
resp. $\omega = \frac {\sqrt \Delta}{2}$ if $\Delta \equiv 0 \pmod 4$.
For any order $\OO$ of discriminant $D$, let $\K_D$ denote
the ring class field that is associated to it. It is well-known that
if $j$ denotes the modular invariant, then
$\K_D = \K(j(c \omega))$; so $\K_D/\K \simeq
\K[X]/(H_D(X))$, where the \textit {class polynomial} $H_D$
is the minimal polynomial of $j (c \omega)$.
It can be used to obtain elliptic curves over finite fields with a
number of points known in advance, with applications to cryptology,
in particular based on the Weil or Tate pairing
(cf.~\cite{FrScTe10}), and primality proving \cite{AtMo93}.

Since the class polynomial has a rather large height, it is
desirable to find smaller defining polynomials to speed up the computations.
There is a long history of such studies, going back to at least Weber
\cite{Weber02-III}; see, e.g., \cite{Birch69,Stark69,Meyer70} for
connections with the class number~$1$ problem.
Generally modular functions $f$ and special arguments $\alpha \in \OO$
are considered such that the \emph {singular value} $f(\alpha)$
lies in $\K_D$, in which case $f (\alpha)$ is called a
\textit {class invariant}.

Our ultimate goal is to build elliptic curves having CM, and this is
done using a so-called modular equation (with integer coefficients)
relating a modular function $f$ to $j$. For this to be efficient, we
need $f (\alpha)$ to have a small height {\em and} the corresponding modular
equation to be of small genus (with a predilection for genus $0$).

Part of the literature has concentrated on the functions
introduced by Weber, quotients of two $\eta$-functions with a transformation
of level~$2$ applied to one of them, see
\cite {Schertz76,Gee99,GeSt98,Schertz02} besides the already cited
sources. This is a perfect case for us, since the genus of the
associated modular curve is~$0$.

Results on more general $\eta$-quotients are given in
\cite{Hajir93,Hajir93b,HaVi97,Gee99,EnSc04,EnSc12}.
All of them are obtained using the modern tool for determining the Galois
action of the class group of~$\OO$ on singular values of modular functions,
namely Shimura's reciprocity law \cite{Shimura71}.
The present article is no exception to this rule.
For the sake of self-containedness and the reader's ease, we briefly
summarise in \S\ref{sct:Shimura} the reciprocity law in the version
of \cite{Schertz02}, which is most suited to actual computations.

In this article, we propose a systematic study of class invariants obtained
as singular values of the \emph {generalised Weber} functions $\w_N$,
defined and studied in \S\ref {sct:wN}, which are quotients of two
$\eta$-functions with a transformation of level~$N$ applied to one of them.
These appear in \cite [Table~1]{Schertz02} and as a special case
of \cite {HaVi97}. While there is some overlapping between this article
and \cite {HaVi97}, we follow a different approach:
The authors of \cite {HaVi97} use an ideal in the class group
to transform the $\eta$-function, and the norm of the ideal implicitly
determines the level; they then proceed to prove which root of unity
is needed for twisting the function so that a minimal power of it yields
a class invariant. On the other hand, we start with a fixed level and thus
a fixed generalised Weber function and determine the minimal power yielding
class invariants without using additional roots of unity.

A first result on the ``canonical'' power $\w_N^s$ is readily obtained
in \S\ref{sct:fullpowers} by a direct application of Shimura reciprocity.
Examining the Galois action on the singular values in \S\ref {sct:lowerbis}
allows us to determine the precise conditions under which lower powers
$\w_N^e$ with $e \mid s$ yield class invariants in
\S\ref {sct:specialization}.

While there is always some transformation level~$N$ (or, equivalently,
an ideal in the class group) such that the corresponding generalised Weber
function yields a class invariant, fixing the level first as we do it in
this study implies control over the height of the class invariants.
Indeed, this height, an important measure for the complexity
of computing a class polynomial,
is asymptotically given as a function of the degrees of the modular
polynomials relating the modular function to the $j$-invariant. 
Thus, the generalised Weber functions can be ordered totally with respect
to their computational efficiency, see \S\ref{sct:heights}, and the
invariants can be compared directly to other invariants in the literature,
cf. \cite{EnMo02,Enge07}.

Unlike \cite {HaVi97}, we explicitly consider levels $N$ that are
not coprime to~$6$, a considerable source of complication, which is justified
since the corresponding functions tend to yield class invariants of lower
height, see the formul{\ae} in \S\ref {ssct:heights} and
Table~\ref {tab:comparison}. Otherwise said, the corresponding modular
curves, related to $2$- and $3$-torsion points on elliptic curves, have a
lower genus than would be expected from the size of~$N$ alone.
This makes it easier to construct the associated elliptic curves with
complex multiplication; in particular, \cite {Morain07} shows how $\w_3$
can be used to directly write down the correct twist of the elliptic curve
with the desired number of points over a finite field.

Existing results in the literature often only state when
a singular value is a class invariant; to obtain the class polynomial,
however, one needs an explicit description of its algebraic conjugates.
These can be worked out using Shimura reciprocity again;
following the approach of $N$-systems introduced in \cite{Schertz02},
we obtain synthetic and simple descriptions of the conjugates, and moreover
determine when the class invariant has a minimal polynomial with rational
coefficients, that is, it defines the real subfield of the class field
over~$\Q$.

\section{Class invariants by Shimura reciprocity}
\label{sct:Shimura}

In the following, we denote by $f \circ M$ the action of matrices
$M = \begin {pmatrix} a & b \\ c & d \end {pmatrix}
\in \Gamma = \Sl_2 (\Z) / \{ \pm 1 \}$
on modular functions given by
\[
(f \circ M) (z) = f (M z) = f \left( \frac {a z + b}{c z + d} \right).
\]
For $n \in \N$, let
$\Gamma (n) = \left\{ \begin {pmatrix} a & b \\ c & d \end {pmatrix}
\equiv \begin {pmatrix} 1 & 0 \\ 0 & 1 \end {pmatrix}
\pmod n \right\}$
be the principal congruence subgroup of level $n$; for a
congruence subgroup $\Gamma'$ such that
$\Gamma (n) \subseteq \Gamma' \subseteq \Gamma$, denote by
$\C_{\Gamma'}$ the field of modular functions for $\Gamma'$.
One of the most important congruence subgroups is given by
$\Gamma^0 (n) = \left\{ \begin {pmatrix} a & b \\ c & d \end {pmatrix}
\equiv \begin {pmatrix} \ast & 0 \\ \ast & \ast \end {pmatrix}
\pmod n \right\}$.

\begin {definition}
\label {def:FN}
The set $\F_n$ of modular functions of level $n$ rational over the
$n$-th cyclotomic field $\Q (\zeta_n)$ is given by all functions
$f$ such that
\begin {enumerate}
\item
$f$ is modular for $\Gamma (n)$ and
\item
the $q$-expansion of $f$ has coefficients in $\Q (\zeta_n)$, that is,
\[
f \in \Q (\zeta_n) \left( \left( q^{1/n} \right) \right),
\]
where $q^{1/n} = e^{2 \pi i z / n}$.
\end {enumerate}
\end {definition}

The function field extension $\F_n / \Q (j)$ has Galois group isomorphic to
$\Gl_2 (\Z / n \Z) / \{ \pm 1 \}$, where the isomorphism is defined by the following action of matrices on functions:
\begin {itemize}
\item
$(f \circ M) (z) = f (M z)$ as above for $M \in \Gamma$;
this implies in particular that also the $q$-expansion of $f \circ M$
has coefficients in $\Q (\zeta_n)$;
\item
$f \circ \begin {pmatrix} 1 & 0 \\ 0 & d \end {pmatrix}$
for $\gcd (d, n) = 1$ is obtained by applying to the $q$-expansion of $f$ the automorphism $\zeta_n \mapsto \zeta_n^d$;
\item
any other matrix $M$ that is invertible modulo $n$ may be decomposed as
$M \equiv M_1 \begin {pmatrix} 1 & 0 \\ 0 & d \end {pmatrix}
M_2 \pmod n$
with $\gcd (d, n) = 1$ and $M_1$, $M_2 \in \Gamma$, and
\[
(f \circ M) (z) = \left( \left( (f \circ M_1) \circ
\begin {pmatrix} 1 & 0 \\ 0 & d \end {pmatrix} \right)
\circ M_2 \right) (z).
\]
\end {itemize}

Shimura reciprocity makes a link between the Galois group of the
function field $\F_n$ and the Galois groups of class fields generated
over an imag\-i\-nary-qua\-drat\-ic field by singular values of modular
functions.

\begin {theorem}[Shimura's reciprocity law, Th.~5 of \cite{Schertz02},
Th.~5.1.2 of \cite{Schertz09}]
\label {th:shimura}
Let $f$ be a function in $\F_n$, $\Delta < 0$ a fundamental discriminant and
$\OO$ the order
of $\K = \Q (\sqrt \Delta)$ of conductor $c$.
In the following, all $\Z$-bases of ideals are written as column vectors.
Let $\af = \begin {pmatrix} \alpha_1 \\ \alpha_2 \end {pmatrix}_\Z$ with
basis quotient $\alpha = \frac {\alpha_1}{\alpha_2} \in \HH$ be a proper
ideal of $\OO$, $\m$ an ideal of $\OK$ of norm $m$ prime to $c n$,
$\overline \m$ its conjugate ideal and $M \in \Gl_2 (\Z)$ a matrix of
determinant $m$ such that
$M \begin {pmatrix} \alpha_1 \\ \alpha_2 \end {pmatrix}$
is a basis of $\af (\overline \m \cap \OO)$. If $f$ does not have a pole
in $\alpha$, then
\begin {itemize}
\item
$f (\alpha)$ lies in the ray class field modulo $c n$ over $K$ and
\item
the Frobenius map $\sigma (\m)$ acts as
\[
f (\alpha)^{\sigma (\m)} = (f \circ m M^{-1}) (M \alpha).
\]
\end {itemize}
\end {theorem}

In the following, we are particularly interested in
\textit {class invariants}, that is, values $f (\alpha)$ that lie not only in a ray class field, but even in a ring class field.
Using Shimura's reciprocity law, \cite[Th.~4]{Schertz02} gives a very
general criterion for class invariants, which is the basis for our
further investigations.

\begin {theorem}
\label {th:main}
Let $f \in \C_{\Gamma^0 (n)}$ for some $n \in \N$
be such that
$f$ itself and $f \circ S$ have rational $q$-expansions. Denote by
$\alpha \in
\HH$ a root of the primitive form $[A, B, C]$ of discriminant $D$ with
$\gcd (A, n) = 1$ and $n \mid C$. If $\alpha$ is not a pole of $f$, then $f
(\alpha) \in \K_D$.
\end {theorem}

The conjugates of $f (\alpha)$ are then derived generically in a form
that is well suited for computations in \cite[Prop.~3 and Th.~7]{Schertz02},
\cite[Th.~5.2.4]{Schertz09}.

\begin {theorem}
\label {th:N-system}
An \textit {$n$-system} for the discriminant $D$ is a complete system of
equivalence classes of primitive quadratic forms
$[A_i, B_i, C_i] = A_i X^2 + B_i X + C_i$,
$i = 1, \dots, h (D)$, of
discriminant $D = B_i^2 - 4 A_i C_i$, such that $\gcd (A_i, n) = 1$ and
$B_i \equiv B_1 \pmod {2 n}$. Such a system exists for any $n$. To these
quadratic forms, we associate in the following the quadratic numbers
$\alpha_i = \frac {-B_i + \sqrt D}{2 A_i}$.

Let $f \in \F_n$ be such that $f \circ S$ with
$S = \begin {pmatrix} 0 & -1 \\ 1 & 0 \end {pmatrix}$
has a rational $q$-expansion.
If $f (\alpha_1) \in \K_D$, then a complete system of conjugates of
$f (\alpha_1)$ under
the Galois group of $\K_D$ is given by the $f (\alpha_i)$, and
the characteristic polynomial of $f (\alpha_1)$ over $\K$ is
\[
H_D [f] = \prod_{i = 1}^{h (D)} (X - f (\alpha_i)).
\]
\end {theorem}

\section{The generalised Weber functions \texorpdfstring {$\w_N$}{wN}}
\label{sct:wN}

In this section we examine the general properties of the
function $\w_N$, with the aim in mind of applying Theorem~\ref
{th:main} to its powers.

Let $z$ be any complex number and put $q=e^{2 i \pi z}$. Dedekind's
$\eta$-function is defined by \cite{Dedekind76}
$$\eta(z) =  q^{1/24} \prod_{m \geq 1} (1 - q^m).$$
The Weber functions are \cite[\S~34, p.~114]{Weber02-III}
$$\wf(z) = \zeta_{48}^{-1}\; \frac{\eta((z+1)/2)}{\eta(z)},
\quad\wfone(z) = \frac{\eta(z/2)}{\eta(z)},
\quad\wftwo(z) = \sqrt{2}\; \frac{\eta(2z)}{\eta(z)}.$$
The modular invariant $j$ is recovered via \cite[\S~54, p.~179]{Weber02-III}:
$$j(z) = \frac{(\wf^{24}-16)^3}{\wf^{24}}
= \frac{(\wfone^{24}+16)^3}{\wfone^{24}}
= \frac{(\wftwo^{24}+16)^3}{\wftwo^{24}}.
$$
The functions $-\wf^{24}$, $\wfone^{24}$ and $\wftwo^{24}$ are
the three roots of the modular polynomial
$$\Phi_2^c(F, j) = {F}^{3}+48\,{F}^{2}+F(768-j)+4096,$$
that describes the curve $X_0(2)$.

For an integer $N>1$, let
the \textit{generalised Weber function} be defined by
$$
\w_N=\frac{\eta(z/N)}{\eta(z)}.
$$
As shown in the following, there is a canonical exponent $t$
such that $\w_N^t$ is modular for $\Gamma^0 (N)$. Its minimal polynomial
$\Phi_N^c(F, j)$ over $\C (j)$ is a model for $X_0(N)$. The other
roots of this polynomial can be expressed in terms of $\eta$, too,
a topic to which we come back in \S\ref{sct:heights}.

We need to know the behaviour of $\w_N$ under
unimodular transformations, which can be broken down to the transformation
behaviour of $\eta (z / K)$ for $K = 1$ or $N$. This has been worked out
in \cite [Th.~3]{EnSc05}.

\begin {theorem}
\label {th:transformation}
Let $M = \begin {pmatrix} a & b \\ c & d \end {pmatrix} \in \Gamma$ be
normalised such that $c \geq 0$, and $d >0$ if $c = 0$. Write $c = c_1
2^{\lambda (c)}$ with $c_1$ odd; by convention, $c_1 = \lambda (c) = 1$ if
$c = 0$. Define
\[
\varepsilon (M) = \legendre {a}{c_1} \zeta_{24}^{a b + c (d (1 - a^2) - a)
+ 3 c_1 (a - 1) + \frac {3}{2} \lambda (c) (a^2 - 1)}.
\]
For $K \in \N$ write
\[
u a + v K c = \delta = \gcd (a, Kc) = \gcd (a, K).
\]
Then
\[
\eta \left( \frac {z}{K} \right) \circ M = \varepsilon \begin {pmatrix}
\frac {a}{\delta} & -v \\ \frac {Kc}{\delta} & u \end {pmatrix} \sqrt
{\delta (c z + d)} \, \eta \left( \frac {\delta z + (u b + v K d)}{\frac
{K}{\delta}} \right),
\]
where the square root is chosen with positive real part.
\end {theorem}

\begin {theorem}
\label {th:F24N}
The function $w_N$ has a rational $q$-expansion.
Denote by $S = \begin
{pmatrix} 0 & -1 \\ 1 & 0 \end {pmatrix}$ the matrix belonging to the
inversion $z \mapsto - \frac {1}{z}$. If $N$ is a square, then $\w_N \circ S$
has a rational $q$-expansion. Otherwise, $\w_N^2 \circ S$ has a rational
$q$-expansion.

Let the subscript 1 and the function $\lambda$ have the same
meaning for a positive integer $n$ as in Theorem~\ref {th:transformation},
that is, $n = n_1 \, 2^{\lambda (n)}$ with $n_1$ odd.
If $M = \begin {pmatrix} a & N b_0 \\ c & d \end {pmatrix}
\in \Gamma^0 (N),$
then $\w_N \circ M = \varepsilon \, \w_N$ with
\begin {equation}
\label {eq:epsilon}
\varepsilon =
\legendre {a}{N_1} \zeta_{24}^{(N - 1) (-b_0 a + c (d (1 - a^2)
- a))} \zeta_4^{c_1 \frac {(N_1 - 1)(a - 1)}{2}} (-1)^{\frac {\lambda (N)
(a^2 - 1)}{8}}.
\end {equation}
Let $t = \frac {24}{\gcd (N - 1, 24)}$ measure how far $N-1$ is from
being divisible by $24$, and let $e$ and $s$ be such that
$t \mid s \mid 24$ and $e \mid s$.
If $N_1$ is a square or $e$ is even, then
$\w_N^e$ is modular for $\Gamma \left( \frac {s}{e} \right)
\cap \Gamma^0 \left( \frac {s}{e} N \right)$.
Otherwise, $\w_N^e$ is modular for $\Gamma \left( \frac {s}{e} N_1 \right)
\cap \Gamma^0 \left( \frac {s}{e} N \right)$.
In both cases, $\w_N^e \in \F_{\frac {s}{e} N} \subseteq \F_{24 N}$.
\end {theorem}

\begin {proof}
The $q$-expansion of $w_N$ is rational since that of $\eta$ is.
Let $M = \begin {pmatrix} a & b \\ c & d \end {pmatrix} \in \Gamma$.
By Theorem~\ref {th:transformation} applied to $K = 1$ and $N$, we have
\begin {equation}
\label {eq:w_N|M}
\w_N \circ M = \varepsilon \begin {pmatrix} \frac {a}{\delta} & -v \\ \frac
{Nc}{\delta} & u \end {pmatrix}
\varepsilon \begin {pmatrix} a & b \\ c & d \end {pmatrix}^{-1}
\sqrt \delta \,
\frac {\eta \left( \frac {\delta z + (ub + vNd)}{\frac {N}{\delta}}
\right)}
{\eta (z)}
\end {equation}
with $\delta = \gcd (a, N) = u a + v N c$.

In the special case $M = S$ we obtain $\delta = N$, $v = 1$, $u = 0$ and
\[
\w_N \circ S = \sqrt N \, \frac {\eta (N z)}{\eta (z)},
\]
which proves the assertion on the $q$-expansion of $\w_N \circ S$.

Assume now that $M \in \Gamma^0 (N)$. Letting $b = N b_0$, we have
$\delta = 1$, $u = d$ and $v = - b_0$ since $a d - b c = 1$.
Thus, (\ref {eq:w_N|M}) specialises as
\[
\w_N \circ M
= \varepsilon \begin {pmatrix} a & b_0 \\ Nc & d \end {pmatrix}
\varepsilon \begin {pmatrix} a & b \\ c & d \end {pmatrix}^{-1}
\frac {\eta (z / N)}{\eta (z)} \\
= \varepsilon \, \w_N (z)
\]
with
\[
\varepsilon = \legendre {a}{c_1 N_1} \legendre {a}{c_1}^{-1}
\zeta_{24}^{(b_0 - b) a + c (N - 1) (d (1 - a^2) - a) + 3 c_1 (N_1 - 1) (a
- 1) + \frac {3}{2} (\lambda (N c) - \lambda (c)) (a^2 - 1)},
\]
which proves \eqref {eq:epsilon}.

We need to examine under which conditions $\varepsilon^e = 1$.
The Legendre symbol vanishes when $N_1$ is a square, $e$ is even
or $a \equiv 1 \pmod {N_1}$. The exponent of $\zeta_{24}$ becomes
divisible by $s (N - 1)$ and thus by $24$ whenever $\frac {s}{e}$
divides $b_0$ and $c$.

In the case of odd $N$, we have $\lambda (N) = 0$ and $N = N_1$, and
the condition on $a$ implies that the exponent of $\zeta_4$ is
divisible by~$4$.

In the case of even $N$, the coefficient $a$ is odd since $\det M = 1$,
and
$$
\varepsilon^e = (-1)^{e \left( c_1 \frac {(N_1 - 1)(a - 1)}{4}
+ \lambda (N) \frac {a^2 - 1}{8} \right)}.
$$
For even $e$, there is nothing to show. If $e$ is odd, then $8 \mid t \mid s$
implies that $a \equiv 1 \pmod 8$, which finishes the proof.
\end {proof}

\section{Full powers of \texorpdfstring {$\w_N$}{wN}}
\label{sct:fullpowers}

To be able to apply Theorem~\ref {th:main} directly to powers of
$\w_N$, we are interested in the minimal exponent $s$ such that $\w_N^s$
is invariant under $\Gamma^0 (N)$ and $\w_N^s \circ S$ has a rational
$q$-expansion. From Theorem~\ref {th:F24N}, we recover the integer $t
= 24/(\gcd (N - 1, 24))$ and recall that
$s = 2t$ if $t$ is odd and $N$ is not a square, and $s = t$
otherwise.

\subsection{Arithmetical prerequisites}

We begin with the following purely arithmetical lemma.
\begin{lemma}\label{lemma:eq2}
Let $N$ be an integer. For a prime $p$, denote
by $v_p$ the $p$-adic valuation.
Let $D = c^2 \Delta$ be a discriminant with fundamental part
$\Delta$. Then $D$ admits a square root $B$ modulo $4 N$
if and only if for each prime $p$ dividing $N$, one
of the following holds.
\begin{enumerate}
\item $\legendre{\Delta}{p}=+1$;
\item $\legendre{\Delta}{p}=-1$ and $v_p(N)\leq 2 v_p(c)$;
\item $\legendre{\Delta}{p}=0$  and $v_p(N) \leq 2 v_p(c)+1$.
\end{enumerate}
\end{lemma}

\begin{proof}
The Chinese remainder theorem allows to argue modulo the
different prime powers dividing $N$. The argumentation is
slightly different for $p$ odd and even, and we give some hints
only for $p=2$.

When $\Delta\equiv 1\bmod 8$, $\Delta$ admits a square root modulo any
power of 2.

When $\Delta$ is even, then $\Delta \equiv 8$ or $12 \pmod {16}$,
and $\Delta$ is a square modulo $8$, but not modulo any higher power
of~$2$. Therefore, $c^2 \Delta$ is a square modulo $4N$ if and only if
$v_2(c^2) + 3 \geq v_2(4N)$.

When $\Delta\equiv 5\bmod 8$, $\Delta$ has a square root modulo
$4$ but not modulo $8$, so that $v_2(c^2) + 2 \geq v_2(4N)$ is needed in
that case.
\end{proof}

In the following, arithmetical conditions on a prime $p$ to be
representable by the principal form of discriminant $D$ will be
needed.
We take the following form of Dirichlet's theorem from
\cite[Ch. 4]{Buell89} (alternatively, see \cite[Chap 18,
G]{Cohn78}). For an integer $p$, let $\chi_4(p) = \legendre{-1}{p}$ and
$\chi_8(p) = \legendre{2}{p}$. The {\em generic characters} of $D =
c^2 \Delta$ are defined as follows:
\begin{alphenumerate}
 \item $\legendre{p}{q}$ for all odd primes $q$ dividing $D$;
 \item if $D$ is even:
  \begin{romanenumerate}
   \item $\chi_4(p)$ if $D/4 \equiv 3, 4, 7\pmod 8$;
   \item $\chi_8(p)$ if  $D/4 \equiv 2\pmod 8$;
   \item $\chi_4(p) \cdot \chi_8(p)$ if  $D/4 \equiv 6\pmod 8$;
   \item $\chi_4(p)$ and $\chi_8(p)$ if  $D/4 \equiv 0\pmod 8$.
  \end{romanenumerate}
\end{alphenumerate}
Note that if $D$ is fundamental (i.e., $c=1$), then case (iv) cannot
occur and in case (i), we may have $D/4 \equiv 3, 7\pmod 8$ only.

\begin{theorem}\label{thm-Dirichlet}
An integer $p$ such that $\gcd(p, 2cD)=1$ is representable by some
class of forms in the principal genus of discriminant $D$ if and only
if all generic characters $\chi(p)$ have value $+1$.
In particular, this condition is necessary for representability
by the principal class.
\end{theorem}

\subsection{Class invariants}

\begin {theorem}
\label {th:full}
Let $N$ be an integer and $t = \frac {24}{\gcd (N - 1, 24)}$. If $t$ is
odd and $N$ is not a square, let $s = 2t$, otherwise, let $s = t$.
Suppose $D$ satisfies Lemma \ref{lemma:eq2}.
Consider an $N$-system of forms $[A_i, B_i, C_i]$ with roots
$\alpha_i = \frac {- B_i + \sqrt D}{2 A_i}$
such that $B_i \equiv B \pmod {2 N}$, as introduced in
\ref{th:N-system}. Then the singular values $\w_N^s (\alpha_i)$ lie
in the ring class field $K_D$, and they form a complete set of Galois
conjugates.
\end {theorem}

\begin {proof}
Once the existence of $B$ is verified, the form $[1, B, C]$ with
$C = \frac {B^2 - D}{4}$ is of discriminant $D$ and satisfies $N \mid C$.
The assertion of the theorem is then a direct
consequence of Theorems~\ref {th:main} and~\ref {th:F24N}.
\end {proof}

Sometimes, the characteristic polynomial of $\w_N^s$ is real, so that its
coefficients lie in $\Z$ instead of the ring of integers of $\Q (\sqrt D)$.
It is then interesting to determine the pairs of quadratic forms that lead to
complex conjugates.

\begin {theorem}
\label {th:reality}
Under the assumptions of Theorem~\ref {th:full}, let $B \equiv 0 \pmod N$,
which is possible whenever $N$ is odd and $N \mid D$,
or $N$ is even and $4 N \mid D$. Then the characteristic polynomial of $\w_N^s$
is real. More precisely, if $\alpha_i$ and
$\alpha_j$ are roots of inverse forms of the $N$-system, then
$\w_N^s (\alpha_j) = \overline {\w_N^s (\alpha_i)}$.
\end {theorem}

\begin {proof}
Notice that $B \equiv 0 \pmod N$ and $B_i \equiv B \pmod {2 N}$
imply $-B_i \equiv B \pmod {2 N}$, so that $[A_i, -B_i, C_i]$,
the inverse form of $[A_i, B_i, C_i]$, satisfies the $N$-system
constraint; thus
$
\w_N^s (\alpha_j)
= \w_N^s \left( \frac {B_i + \sqrt {D}}{2 A_i} \right)
= \w_N^s (\overline {- \alpha_i}).
$
On the other hand, $q (\overline {- \alpha_i}) = \overline {q (\alpha_i)}$,
which implies $\w_N (\overline {-\alpha_i}) = \overline {\w_N (\alpha_i)}$
since $\w_N$ has a rational $q$-expansion.
\end {proof}

These first results, direct consequences of the Shimura reciprocity law, are
meant to set the stage for the detailed and much more involved analysis of
lower powers in the following chapters.
For $\gcd (N, 6) = 1$, \cite [Theorem~20]{HaVi97} determines a
$48$-th root of unity~$\zeta$ and an exponent $e \mid s$
such that $\zeta \w_N^e$ yields a class invariant. With a bit of work,
it can be shown that $\zeta^{s / e} = 1$ in our context, which provides
an alternative proof of Theorem~\ref {th:full} without giving the algebraic
conjugates of the singular value.

\section{Explicit Galois action}
\label{sct:lowerbis}

Throughout the remainder of this section, we assume that $N$ is a
square or $e$ is even, so that $f = \w_N^e$ and $f \circ S$ have
rational $q$-expansions by Theorem~\ref {th:F24N}. Let $\alpha$ be
a root of the primitive quadratic form $[A, B, C]$ of discriminant
$D$ with $\gcd (A, N) = 1$.
By Theorems~\ref {th:F24N} and~\ref {th:shimura}, the singular
value $f (\alpha)$ lies in the ray class field modulo
$c \frac {t}{e} N$ over $\K$, and the Galois action of ideals
in $\OK$ can be computed explicitly. We eventually need to show that
the action of principal prime ideals generated by elements in
$\OO$ is trivial, which implies that the singular value lies
in the ring class field $\K_D$. Then Theorems~\ref {th:F24N}
and~\ref{th:N-system} show that the conjugates
are given by the singular values in a $\frac {t}{e}N$-system.

We are only interested in the situation that $N \mid C$. Notice
that under $\gcd (A, N) = 1$ this is equivalent to
$4 N \mid 4 A C = B^2 - D$, or $B^2 \equiv D \pmod {4 N}$.
The remainder of this section is devoted to computing in
this case
the Galois action of principal prime ideals $(\pi)$ with
$\pi \in \OO$ coprime to $6 c N$ on the
singular values according to the arithmetic properties of $N$ and
$D$. \S\ref {sct:specialization} applies
these results to the determination of class invariants.

To apply Shimura reciprocity in the formulation of
Theorem~\ref {th:shimura}, we need to explicitly write down adapted
bases for the different ideals. So let
$\af = \begin {pmatrix} A \alpha \\ A \end {pmatrix}_\Z$ be an ideal of
$\OO = \begin {pmatrix} A \alpha \\ 1 \end {pmatrix}_\Z$ with basis quotient
$\alpha$.
Without loss of generality, we may assume that $p = \Norm (\pi) \mid C$
by suitably modifying
$\alpha$: Indeed, notice that the quadratic form associated to
$\alpha' = \alpha - 24 k N$ for some $k \in \Z$ is given by
$[A, B', C'] = [A, B + 2 A (24 k N), A (24 k N)^2 + B (24 k N) + C]$.
This form still satisfies $N \mid C'$, and furthermore
$f (\alpha') = f (\alpha)$ since $f$ is invariant under translations
by $24 N$ according to Theorem~\ref {th:F24N}.
Since $p$ splits in $\OO$ and is prime to $c$, the equation $A X^2 + B X + C$
has a root $x$ modulo $p$. Choosing $k \in \Z$ such that
$k \equiv x (24 N)^{-1} \pmod p$, which is
possible since $p \nmid 6 N$, we obtain $p \mid C'$.

Let $\pi = u + v A \alpha$ with $u$, $v \in \Z$. From
\begin{equation}
\label{eq:cond}
p = \Norm (\pi) = u (u - v B) + v^2 A C
\end{equation}
and $p \mid C$ we deduce that $p$ divides $u$ or $u' = u - v B$.
Using $A \overline \alpha = - A \alpha - B$ and
$N (A \alpha) = AC$, we compute
\[
\overline \p \af = \overline \pi \begin {pmatrix} A \alpha \\ A \end {pmatrix}
= \begin {pmatrix} u A \alpha + v A C \\ u A - v A^2 \alpha - v A B \end {pmatrix}
= \begin {pmatrix} u & v C \\ - v A & u - v B \end {pmatrix}
\begin {pmatrix} A \alpha \\ A \end {pmatrix}
\]
So if $p \mid u$, the matrix $M$ of Theorem~\ref {th:shimura} is given by
\[
M = \begin {pmatrix} u & v C \\ - v A & u - v B \end {pmatrix}
=
\begin {pmatrix} p & 0 \\ 0 & 1 \end {pmatrix}
M_1
\text { with }
M_1 =
\begin {pmatrix} \frac {u}{p} & v \frac {C}{p} \\ - v A & u' \end {pmatrix}
\in \Gamma^0 (N)
\]
since $N \mid C$ and $p \nmid N$.

If $f$ is invariant under $M_1^{-1}$, the rationality of its
$q$-expansion implies that
\[
f \circ m M^{-1}
= f \circ M_1^{-1} \circ \begin {pmatrix} 1 & 0 \\ 0 & p \end {pmatrix}
= f,
\]
so that
\[
f (\alpha)^{\sigma (\p)}
= f (M \alpha)
= f \left( \frac {u \alpha + v C}{- v A \alpha + u - v B} \right)
= f \left( \frac {\overline \pi \alpha}{\overline \pi} \right)
= f (\alpha).
\]

For $p \mid u'$, we decompose in a similar manner
\[
M = M_2 \begin {pmatrix} 1 & 0 \\ 0 & p \end {pmatrix}
= M_2 S \begin {pmatrix} p & 0 \\ 0 & 1 \end {pmatrix} S
\text { with }
M_2 = \begin {pmatrix} u & v \frac {C}{p} \\
-v A & \frac {u'}{p} \end {pmatrix}
\in \Gamma^0 (N),
\]
and the rationality of the $q$-expansion of $f \circ S$ allows
to conclude if $f$ is invariant under $M_2^{-1}$.

So we need the transformation of $f$ under
$$M_1^{-1} = \begin {pmatrix} u' & -v \frac {C}{p} \\
v A & \frac {u}{p} \end {pmatrix}.$$
Rewriting \eqref {eq:epsilon}, it is given by
$f \circ M_1^{-1} = \zeta_{24}^{e \theta} f$ with
\begin{equation}\label{eq:theta}
\begin {split}
\theta =
& (N-1) v \left(u' \frac{C}{Np} + A \left(\frac{u}{p} (1-{u}'^2)
- u'\right)\right) \\
& + 3 v_1 A_1 (N_1-1) (u'-1) + \frac {3 \lambda (N)
({u'}^2-1)}{2}.
\end {split}
\end{equation}
We obtain invariance provided $e \theta\equiv 0\bmod 24$. (The
treatment of $M_2^{-1}$ is completely analogous and omitted.)
In the following, we classify the
values of $D$ and $B$ for which $\theta$ is $0$
modulo some divisor of $24$.
It is natural to study separately
$\theta\bmod 3$ and $\theta\bmod 2^{\xi}$ for $1\leq \xi \leq 3$
depending on the value of $N$. We will give code names to the
following propositions for future use.

\subsection{The value of \texorpdfstring {$\theta$}{theta} modulo
\texorpdfstring {$3$}{3}}
\label {ssct:theta3}

To be able to use some exponent $e$ not divisible by~$3$, we
need to impose $3 \mid \theta$. From the reduction of (\ref{eq:theta})
modulo $3$, namely
$$\theta = (N-1) v \left(u' \frac{C}{Np} + A \left(\frac{u}{p} (1-{u}'^2)
- u'\right)\right) \bmod 3,$$
we immediately see that $3 \mid \theta$ for $N\equiv 1\bmod 3$
without any further condition, which is coherent with $3 \nmid s$ in this
case.

For $N \not\equiv 1 \pmod 3$, we impose $B^2 \equiv D \pmod {4N}$ to
obtain divisibility of $C$ by $N$ (see the discussion above), and
define $r \in \{ 0, 1, 2 \}$ such that
\begin {equation}
\label {eq:r3}
A \frac {C}{N} = \frac {B^2 - D}{4 N} \equiv r \pmod 3.
\end {equation}
Notice that $r = 1$ implies $A \equiv \frac {C}{N} \pmod 3$,
while $r = 2$ implies $A \equiv - \frac {C}{N} \pmod 3$.

\subsubsection{The case \texorpdfstring {$N\equiv 0\bmod 3$}{N=0 mod 3}}

\begin{proposition}[PROP30]\label{prop:Neq0mod3}
Let $N \equiv 0 \pmod 3$, $B^2 \equiv D \pmod {4 N}$
and $r$ as in \eqref {eq:r3}.
Then $3 \mid \theta$ if
\begin {alphenumerate}
\item
$3 \mid D$ and $r = 1$;
\item
$D \equiv 1 \pmod 3$ and $r=2$.
\end {alphenumerate}
In these cases, $B$ satisfies the following congruences modulo~$3$:
\begin {alphenumerate}
\item
$3 \mid B$;
\item
$3 \nmid B$.
\end {alphenumerate}
\end{proposition}

\begin{proof}
Since $3 \mid N \mid C$ and $3 \nmid p$,
$u^2 \equiv {u'}^2\equiv 1\pmod 3$ by \eqref {eq:cond} and
\[
\theta \equiv \pm v \left(\frac{C}{Np} - A\right) \bmod 3.
\]

\begin {alphenumerate}
\item
If $3 \mid B$, or equivalently $3 \mid D$, then $p\equiv u^2 \equiv 1\pmod 3$
in \eqref{eq:cond}. The desired result follows from~\eqref {eq:r3}.
\item
If $3 \nmid B$, which is equivalent with $D \equiv 1 \pmod 3$,
only the case $3 \nmid v$ needs to be examined. Then
$u \not\equiv u' \pmod 3$
and $p \equiv 2 \pmod 3$, and again \eqref {eq:r3} allows to
conclude.
\end {alphenumerate}
\end{proof}

\subsubsection{The case \texorpdfstring {$N\equiv 2\bmod 3$}{N=2 mod 3}}

\begin{proposition}[PROP32]\label{prop:Neq2mod3}
Let $N \equiv 2 \pmod 3$, $B^2 \equiv D \pmod {4 N}$
and $r \in \{1, 2\}$ as in \eqref {eq:r3}.
If $D \equiv r \pmod 3$, then $3 \mid \theta$ and $3 \mid B$.
\end{proposition}

\begin{proof}
Notice that $D \equiv r \pmod 3$ is equivalent with
$3 \mid B$ by \eqref {eq:r3}.
Then $u'\equiv u\pmod 3$ and
$$\theta \equiv u v \left(\frac{C}{Np} + \frac{A}{p} (1-u^2)
- A\right) \pmod 3.$$
If $3$ divides $u$ or $v$, we are done.

Otherwise, $u^2 \equiv v^2 \equiv 1\pmod 3$,
which implies
$$\theta \equiv \pm \left(\frac{C}{Np} -A\right) \pmod 3.$$
Writing $p \equiv 1 + A C \equiv 1 - r \pmod 3$, we see that this case
is possible only for $r = 2$ and $p \equiv 2 \pmod 3$,
and then $A \equiv - \frac {C}{N} \pmod 3$ and $3 \mid \frac{C}{Np} - A$.
\end {proof}

Note that the proposition does not hold for $r=0$, since then $3 \mid D$,
$3 \mid B$, $3 \mid AC$, and exactly one of $A$ and $C$ is divisible by $3$
(if both were, then $[A, B, C]$ would not be primitive), causing
$\theta\not\equiv 0\bmod 3$ unless one of $u$ or $v$ is divisible by~$3$.

\subsection{The value of \texorpdfstring {$\theta$}{theta} modulo
powers of \texorpdfstring {$2$}{2}}
\label {ssct:theta2}

\subsubsection{The case \texorpdfstring {$N$}{N} odd}
\label{ssct:Nodd}

Since $N_1=N$ and $\lambda (N) = 0$, \eqref{eq:theta} becomes
$$\theta \equiv (N-1) \rho \pmod 8$$
for
$$\rho = v \left(u' \frac{C}{Np} + A \left(\frac{u}{p} (1-{u}'^2)
- u'\right)\right) + 3 v_1 A_1 (u'-1).$$
So $\theta$ is divisible by~$8$ if $N \equiv 1\pmod 8$,
which is the case in particular if $N$ is a square.
Otherwise, $e$ is supposed to be even, so $e \theta$ is divisible
by $4$; if $N \equiv 1 \pmod 4$, $e \theta$ is even divisible by~$8$.
So the only remaining case of interest is $N \equiv 3 \pmod 4$;
then for $e \equiv 2 \pmod 4$, $8 \mid e \theta$ is equivalent with
$\rho$ even. We have
$$\rho \equiv v \big( u' C + A (u (1+u')+u') \big) + u'+1 \bmod 2.$$

\begin{proposition}[PROP21]\label{prop:Dodd}
Let $N$ be odd. If $D$ is odd, then $\theta \equiv (N-1) \rho \pmod 8$
with $\rho$ even.
\end{proposition}

\begin{proof}
Since $B$ is odd, $u' \equiv u+v \pmod 2$.

If one of $v$, $A$ and $C$ is even, then $u$ and $u'$ are odd by
\eqref {eq:cond} (so that in fact $v$ is even), and $\rho$ is even.

Otherwise, $v$, $A$ and $C$ are odd, $u' = u+1\pmod 2$ and
$\rho$ is even as well.
\end{proof}

\subsubsection{The case \texorpdfstring {$N$}{N} even}

Let $N = 2^{\lambda (N)} N_1$ with $N_1$ odd and
$\lambda (N) \geq 1$. We study divisibility of $\theta$ by
$2^{\xi}$ for increasing values of $\xi$. The value $\xi = 3$ is
of interest only when $e$ is odd, in which case $N$ and thus $N_1$
are squares. We start with an elementary remark.

\begin{lemma}\label{lem:thetamod4}
If $2 \mid N\mid C$, then
\begin {alphenumerate}
\item
$u$ and $u'$ are odd and
\begin{equation}\label{eq:thetamod4}
\theta \equiv (N-1) v u' \left(\frac{C}{Np} - A\right) \pmod 4;
\end{equation}
\item
moreover, if $4\mid C$, then $2\mid v B$.
\end {alphenumerate}
\end{lemma}

\begin{proof}
\begin {alphenumerate}
\item
$u$ and $u'$ are odd by \eqref {eq:cond},
so that ${u'}^2\equiv 1\pmod 8$. Since $N_1$ is odd, almost all terms
disappear from \eqref{eq:theta}.
\item
We have $p = u^2 + v (-u B + v A C)\equiv u (u - v B) \bmod
4$. Since $u$ is odd by (a), we deduce that $v B$ must be even.
\end {alphenumerate}
\end{proof}

As discussed above, $N \mid C$ is equivalent with
$B^2 \equiv D \pmod {4 N}$. Then
$A \frac {C}{N} = \frac {B^2 - D}{4 N}$; by gradually imposing more
restrictions modulo powers of~$2$ times $4 N$, we fix $A \frac {C}{N}$
modulo powers of~$2$.

\begin{proposition}[PROP20]\label{prop:Neventheta2}
When $N$ is even, $\theta$ is even in the following cases:
\begin {alphenumerate}
\item
$B^2 \equiv D + 4N \pmod {8N}$;
\item
$B^2 \equiv D \pmod {8N}$ and $D \equiv 1 \pmod 8$.
\end {alphenumerate}
\end{proposition}

\begin{proof}
\begin {alphenumerate}
\item
The conditions imply that $A (C/N)$ is odd, and
Lemma~\ref{lem:thetamod4}(a) allows to conclude since $p$ is odd.
\item
In that case $A (C/N)$ is even. Since $A$ is prime to $N$, it is
odd and therefore $C/N$ is even, which implies in turn $4 \mid C$.
By Lemma~\ref{lem:thetamod4}(b), we get $2 \mid v B$. Since
$D$ is odd, $B$ is odd and $v$ is even, and (\ref{eq:thetamod4})
finishes the proof.
\end {alphenumerate}
\end{proof}

\subsubsection*{Divisibility of \texorpdfstring {$\theta$}{theta} by $4$}

We begin with a purely arithmetical lemma that will give us necessary
conditions on the parameters for the equation $B^2\equiv D + r (4N)
\bmod (16N)$ to have a solution.
\begin{lemma}
\label{lem:DB4}
Let $r \in \{0, 1, 2, 3\}$ and $N$ be even. Given $D$,
suppose the equation $B^2 \equiv D + 4 r N \pmod {16 N}$ admits a
solution in $B$. Then either $D\equiv 1\bmod 8$ which implies $B$ is
odd, or $D$ is even and $D$ satisfies
one of the conditions of the following table depending on $rN\bmod 8$,
which in turn gives properties of $B$.
$$\begin{array}{|c|c|c|c|}\hline
r N \bmod 8 & \text{condition on } D & \Rightarrow D/4 \bmod 8 & B/2 \\ \hline
0 & 4 \bmod 32   & 1 & \text{odd} \\
  & 16 \mid D    & 0 &\text{even} \\ \hline
2 & 24 \bmod 32  & 6 & 0\bmod 4 \\
  & 28 \bmod 32  & 7 & \text{odd} \\
  &  8 \bmod 32  & 2 & 2 \bmod 4 \\ \hline
4 & 16 \mid D    & 0 & \text{even} \\
  & 20\bmod 32   & 5 & \text{odd} \\ \hline
6 & 8 \mid\mid D & 0 & 0\bmod 4 \\
  & 12 \bmod 32  & 3 & \text{odd} \\ \hline
\end{array}$$
\end{lemma}

\begin{proof}
Since $B^2 \equiv D \bmod 8$, the only possible value for odd $D$ is
$D \equiv 1\bmod 8$, giving $B$ odd. If $D$ is even, then
$$\left(\frac{B}{2}\right)^2 \equiv \frac{D}{4} + r N \bmod 8$$
and since $N$ is even, the above table makes sense.

Remembering that the only squares modulo 8 are $\{0, 1, 4\}$, the
table is easily constructed and left as an exercise to the reader.
\end{proof}

Now, we are ready to extend the result of
Proposition~\ref{prop:Neventheta2} by considering $B^2 \equiv D + r (4
N) \pmod {16 N}$ with $r \in \{ 1, 3 \}$, which yields $A \frac{C}{N}
\equiv r\pmod 4$. Note that case (b) cannot be extended and we leave
the proof of this to the reader.

\begin{proposition}[PROP44]\label{prop:44}
Let $N$ be even, and suppose $B^2 \equiv D + 4 N \pmod {16 N}$ has a solution.
Then $\theta$ is divisible by $4$
if one of the following conditions is met:
\begin {alphenumerate}
\item
$D \equiv 1 \pmod 8$;
\item
$16\mid D$;
\item
$2\mid\mid N$ and $4\mid\mid D$.
\end {alphenumerate}
\end{proposition}

\begin{proof}
If $D$ is odd, the condition follows from Lemma~\ref{lem:DB4}.
Then $u' = u - v B$ leads to $2\mid v$ and $4 \mid \theta$.

Assuming $D$ even, Theorem~\ref{thm-Dirichlet} implies that
$\chi_4(p) = 1$ (or, equivalently, $p\equiv 1\pmod 4$) when $D/4 \bmod 8
\in \{3, 4, 7, 0\}$, which
immediately settles case (b). When $D/4$ is odd, we see that we cannot
have the case $4\mid N$ when comparing with the table of Lemma
\ref{lem:DB4}, and this gives us (c).

In the other cases, when $p\equiv 3\bmod 4$, we get $v$ odd since
$AC\equiv 2\bmod 4$ and there is no reason to have $\theta\equiv
0\bmod 4$.
\end{proof}

\begin{proposition}[PROP412]\label{prop:412}
Let $N$ be even, and suppose $B^2 \equiv D + 12 N \pmod {16 N}$.
Then $\theta$ is divisible by $4$
if one of the following conditions is met:
\begin {alphenumerate}
\item
$D \equiv 1 \pmod 8$;
\item
$8 \mid\mid D$ and $2 \mid\mid N$;
\item
$4 \mid\mid D$ and $4 \mid N$.
\end {alphenumerate}
In the cases of $D$ even, $B$ satisfies the following congruences modulo~$4$:
\begin {alphenumerate}
\setcounter {enumi}{1}
\item
$4 \mid B$;
\item
$2 \mid\mid B$.
\end {alphenumerate}
\end{proposition}

\begin{proof}
The proof for $D$ odd as well as the case distinctions for $D$ even are
the same as in Proposition~\ref {prop:44}. However, we now have
$A \frac {C}{N} \equiv -1 \pmod 4$.

In the cases where $\chi_4(p) = 1$ (i.e., $D/4 \in \{3, 4, 7, 0\}$),
we get $p \equiv 1 \pmod 4$ and $\frac {C}{N p} - A \equiv 2
\pmod 4$. Since there is no compelling reason why $v$ should be even,
$\theta$ may or may not be divisible by $4$.

So we have to turn our attention to the four other cases, i.e., $D/4
\in \{1, 2, 5, 6\}$, with Lemma~\ref{lem:DB4} in mind.
If $4 \mid B$, $8 \mid\mid D$ and $2 \mid\mid N$, then $2 \mid\mid C$,
and either $v$ is even or $p \equiv 3 \pmod 4$. In both cases,
Lemma~\ref {lem:thetamod4} shows that $4 \mid \theta$.
If $2 \mid\mid B$ and $4 \mid\mid D$, suppose that furthermore $4 \mid N$.
Then $4 \mid A C$, and again $v$ is even or $p \equiv 3 \pmod 4$.
\end {proof}


\subsection* {Divisibility of \texorpdfstring {$\theta$}{theta} by $8$}

As discussed at the beginning of \S\ref {ssct:Nodd}, for
generating class fields we are only interested in $\theta \bmod 8$
when $N$ is a square, that is, $\lambda (N)$ is even and $N_1$ is
a square; in particular, $N_1 \equiv 1 \pmod 8$. Then the
following generalisation of Lemma~\ref {lem:thetamod4} is
immediately seen to hold:

\begin {lemma}
\label {lem:thetamod8}
If $N$ is an even square dividing $C$, then
\[
\theta \equiv (N-1) v u' \left(\frac{C}{Np} -A \right) \pmod 8.
\]
\end {lemma}

From the results obtained for $B^2 \equiv D + 4 r N \pmod {16 N}$ for
$r \in \{1, 3\}$, it is natural to look at $B^2\equiv D+ 4 r N\pmod {32N}$
for $r\in \{1, 3, 5, 7\}$. Then $A \frac{C}{N} \equiv r \pmod 8$.

\begin{proposition}[PROP8]
Let $N$ be an even square, and suppose $B^2\equiv D+ 4 r N\pmod {32N}$.
Then $\theta$ is divisible by $8$ if one of the following conditions holds:
\begin {alphenumerate}
\item
$r=3$ or $r = 7$, and $D \equiv 1 \pmod 8$;
\item
$r=1$, and $32 \mid D$;
\item
$r=5$, and $16 \mid\mid D$.
\end {alphenumerate}
In the cases of $D$ even, $B$ satisfies the following congruences modulo~$8$:
\begin {alpharabenumerate}[b]
\item
$4 \mid\mid B$ if $4 \mid\mid N$;
\item
$8 \mid B$ if $16 \mid N$.
\end {alpharabenumerate}
\begin {alpharabenumerate}[c]
\item
$4 \mid\mid B$ if $16 \mid N$;
\item
$8 \mid B$ if $4 \mid\mid N$.
\end {alpharabenumerate}
\end{proposition}

\begin{proof}
Since $4 \mid N \mid C$, we have $p \equiv u (u - v B) \pmod 4$ by
\eqref {eq:cond}.

For $D$ odd, $B$ is odd and $v$ is even as seen in
Proposition~\ref {prop:44}. If $v$ is divisible by $4$, then
$\theta$ is divisible by~$8$ by Lemma~\ref {lem:thetamod8}. If
$2 \mid\mid v$, then $p \equiv 3 \pmod 4$; if furthermore $r \equiv 3 \pmod 4$,
then $4 \mid \frac {C}{N p} - A$, and $8 \mid \theta$ by
Lemma~\ref {lem:thetamod8}.

In the remaining cases of the proposition, $16 \mid D$, $4 \mid B$,
$r \equiv 1 \pmod 4$
and $p \equiv 1 \pmod 4$. If $v$ is even, Lemma~\ref {lem:thetamod8}
implies that $8 \mid \theta$. From now on, we assume that $v$ is odd.
Then $p = u^2 - u v B + A C \pmod 8$, and we need to verify that
$8 \mid \frac {C}{N p} - A$.

The results now follow from close inspection of
$$AC \equiv r N \pmod 8
\text { and }
\left( \frac {B}{4} \right)^2 \equiv \frac {D}{16} + r \frac {N}{4}
\pmod 8.$$

Consider first the case $r = 1$ and $32\mid D$. By
Theorem~\ref{thm-Dirichlet}, we have $\chi_4(p)=\chi_8(p)=1$,
which yields $p\equiv 1\bmod 8$ and implies the desired divisibility
of $\frac {C}{N p} - A$ by~$8$.

Consider now $r = 5$; it is sufficient to show that $p \equiv 5 \pmod 8$.
If $16 \mid\mid D$ and $16 \mid N \mid C$, then $B \equiv 4 \pmod 8$
and $p \equiv 5 \pmod 8$. If $16 \mid\mid D$ and $4 \mid\mid N$, then
$AC \equiv 4 \pmod 8$ and $32 \mid D + 4 r N$, whence $8 \mid B$ and
$p \equiv 5 \pmod 8$.
\end{proof}

\section{Lower powers of \texorpdfstring {$\w_N$}{wN}}
\label{sct:specialization}

The aim of this section is to determine conditions under which singular
values of lower powers of $\w_N$ than those given in Theorem~\ref {th:full}
yield class invariants. When $N$ is not a square, only even powers are
possible by Theorems~\ref {th:F24N} and~\ref {th:main}. So we specialise
the propositions of \S\ref{sct:lowerbis} according to the value of
$N \pmod {12}$. When $N$ is a square, odd powers may yield class invariants,
and we need to distinguish more finely modulo~$24$. Note that then
$N \in \{0, 1, 4, 9, 12, 16 \} \pmod {24}$.

Throughout this section, we use the notation of Theorem~\ref {th:full}.
The number $\alpha$ is a root of the quadratic form $[A, B, C]$ of
discriminant $D$ and $N$ is an integer such that $A$ is prime to $N$ and
$B$ is a square root of $D$ modulo~$4 N$ according to
Lemma~\ref {lemma:eq2}, so that $N \mid C$. The canonical power
$s$ such that $\w_N^s (\alpha)$ is a class invariant, that is,
$\w_N^s (\alpha) \in \K_D$, is defined as
in Theorem~\ref {th:full}, and we wish to determine the minimal
exponent $e$ such that $\w_N^e (\alpha)$ is still a class invariant.
The general procedure is as follows: Given the value of $N$,
we assemble the propositions of \S\ref{sct:lowerbis} (using
their code names throughout) and deduce from them conditions on $B$
as well as the period of $D$ for which class invariants are
obtained. In general, we can combine a condition on $B$ related to
$\theta\bmod 3$ and another one related to $\theta\bmod 2^{\xi}$. The
Chinese remainder theorem is then used to find compatible values.
When no particular condition modulo~$3$ or powers of~$2$ is imposed,
that is, $e$ and $s$ have the same $3$-adic or $2$-adic valuation,
then Theorem~\ref {th:full} already leads to the desired conclusion.

Once a power $\w_N^e (\alpha)$ is identified as a class invariant,
its conjugates may be obtained by an $M$-system for $M = \frac {s}{e} N$
containing $[A, B, C]$ as shown through Theorems~\ref {th:N-system}
and~\ref {th:F24N}. In more detail, one may proceed as follows:
\begin {enumerate}
\item
Determine a form $[A, B, C]$ with root $\alpha$ satisfying
$\gcd (A, M) = 1$ and the constraint on $B$ so that
$\w_N^e (\alpha)$ is a class invariant; in general, one may
choose $A = 1$.
\item
Enumerate all reduced forms $[a_i, b_i, c_i]$, $i = 1, \ldots, h (D)$
of discriminant $D$, numbered in such a way that
$[a_1, b_1, c_1] \equiv [A, B, C]$.
\item
Let $[A_1, B_1, C_1] = [A, B, C]$.
For $i \geq 2$, find a form $[A_i, B_i, C_i] \equiv [a_i, b_i, c_i]$
such that $\gcd (A_i, M) = 1$ and $B_i \equiv B \pmod {2 M}$, using,
for instance, the algorithm of \cite[Prop.~3]{Schertz02},
\cite[Th.~3.1.10]{Schertz09}.
\end {enumerate}
Then a floating point approximation of the class polynomial can be
computed as
\[
\prod_{i = 1}^{h_D} \big( X - \w_N^e (\alpha_i) \big)
\]
with $\alpha_i = \frac {- B_i + \sqrt D}{2 A_i}$.
Using the algorithms of \cite{Enge09}, one obtains a quasi-linear
complexity in the total size of the class polynomial.

Note that the conditions on $B$ of \S\ref {sct:lowerbis} can be
summarised as $B^2 \equiv D + 4 r N \pmod {4 R N}$, where $r$ is
defined modulo~$R$ and the only primes dividing $R$ are $2$ and $3$.
For the sake of brevity, we denote such a
condition by $r$:$R$. So if no particular condition
beyond $B^2 \equiv D \pmod {4 N}$ is required, this is denoted by
0:1.

We will give more details for the first non-trivial cases and be
briefer in the sequel, since the results rapidly become unweildy. We
add numerical examples for these cases.

\subsection{The case \texorpdfstring {$N$}{N} odd}
\label {sct:Nodd}

\subsubsection{\texorpdfstring {$N\not\equiv 0\bmod 3$}{N!=0 mod 3}}

This is the simplest case. We may use PROP32, PROP21 or both
of them.
Whenever $N \equiv 2 \pmod 3$ and $3 \nmid D$, then PROP32 applies;
moreover, the resulting condition $3 \mid B$ is automatically satisfied,
and we gain a factor of $3$ in the exponent. Similarly if $D$ is odd,
then PROP21 applies without any restriction on $B$, and we gain a factor
of $2$ in the exponent.

\begin{center}
\begin{tabular}{|r|r||l|l||r||l|}\hline
$N \bmod {12}$ &  $s$ & $B$   & $D$               & $e$ & proposition(s) \\ \hline
$  5$          &  $6$ & 1:3 & $D\equiv 1\bmod 3$       & $2$ & PROP32 \\
$  5$          &  $6$ & 2:3 & $D\equiv 2\bmod 3$       & $2$ & PROP32 \\
\hline
$  7$          &  $4$ & 0:1   & $2 \nmid D$       & $2$ & PROP21 \\
\hline
$ 11$          & $12$ & 0:1   & $2 \nmid D$       & $6$ & PROP21 \\
$ 11$          & $12$ & 1:3 & $D\equiv 1\bmod 3$       & $4$ & PROP32 \\
$ 11$          & $12$ & 2:3 & $D\equiv 2\bmod 3$       & $4$ & PROP32 \\
$ 11$          & $12$ & 1:3 & $D\equiv 1\bmod 6$ & $2$ & PROP32+PROP21 \\
$ 11$          & $12$ & 2:3 & $D\equiv 5\bmod 6$ & $2$ & PROP32+PROP21 \\
\hline
\end{tabular}
\end{center}
Letting $D = c^2\Delta$, we put $\omega = \sqrt{\Delta/4}$ if $4 \mid
\Delta$ and $\omega = (1+\sqrt{\Delta})/2$ otherwise. Here are some
numerical examples:
$$\begin{array}{|r|c|r|l|}\hline
N & f & -D & H_{D}[f] \\ \hline
5 & \w_{5}^{2} & 11 &
X-\omega-1
\\
5 & \w_{5}^{2} & 4 &
X-1-2\,\omega
\\
\hline
7 & \w_{7}^{2} & 3 &
X-3\,\omega+1
\\
\hline
11 & \w_{11}^{6} & 39 &
{X}^{4}+ \left( 27\,\omega-73 \right) {X}^{3}+ \left( 1656\,\omega-
8914 \right) {X}^{2}\\
&&&+ \left( 7947\,\omega-139058 \right) X-515016\,
\omega+1000693
\\
11 & \w_{11}^{4} & 8 &
X+7+6\,\omega
\\
11 & \w_{11}^{4} & 28 &
X+8\,\omega-7
\\
11 & \w_{11}^{2} & 11 &
X-2\,\omega+1
\\
11 & \w_{11}^{2} & 7 &
X-2\,\omega+3
\\
\hline
\end{array}$$

\subsubsection{The case \texorpdfstring {$N\equiv 3\pmod {12}$}{N=3 mod 12}}

The situation becomes more intricate when $\gcd (N, 6) \neq 1$.
For $N \equiv 3 \pmod {12}$, we have $s=12$, and $N$ cannot be a square.
Therefore we need an even exponent~$e$.
Since already the full power $\w_N^{12}$ can only be used when
$D$ is a square modulo $4N$, we only have to consider
$D \in \{ 0, 1, 4, 9 \} \pmod {12}$. Then PROP30 applies; moreover,
PROP21 applies whenever $D$ is odd, resulting in the following table.

\begin{center}
\begin{tabular}{|r|r||l|l||r||l|}\hline
$N\bmod {12}$ &  $s$ & $B$ & $D \bmod 12$ & $e$ & propositions(s)  \\
\hline
$3$           & $12$ & 0:1 & $1, 9$       & $6$ & PROP21           \\
$3$           & $12$ & 1:3 & $0, 9$       & $4$ & PROP30(a)        \\
$3$           & $12$ & 2:3 & $1, 4$       & $4$ & PROP30(b)        \\
$3$           & $12$ & 1:3 & $9$          & $2$ & PROP30(a)+PROP21 \\
$3$           & $12$ & 2:3 & $1$          & $2$ & PROP30(b)+PROP21 \\
\hline
\end{tabular}
\end{center}

The entries in the first and last line for $D \equiv 1 \pmod {12}$
may seem redundant; but note that they induce differently severe
restrictions on $B$. The entry $D \equiv 1 \pmod {12}$ in the third
line, as well as $D \equiv 9 \pmod {12}$ in the second line, are
redundant, however: Since PROP21 does not induce any additional
restriction on $B$, the lower exponent is available for precisely
the same quadratic forms. In the following, we will present only
tables that have been reduced accordingly.

However, the previous table does not yet contain the full truth.
A line in the table means that if there is a solution to
$B^2 \equiv D + 4rN \pmod {4RN}$ with $D$ in the given residue
class $D_0$ modulo $12$, then $\w_N^e$ yields a class invariant.
Examining this equation modulo the part of $4 R N$ that contains
only $2$ and $3$ yields further restrictions. Write $N = N_6 N'$
such that the only primes dividing $N_6$ are $2$ and $3$ and
$\gcd (N', 6) = 1$. Then we need to ensure that $D + 4 r N \equiv D$
is a square modulo $N'$; this is guaranteed by Lemma~\ref {lemma:eq2},
since otherwise we would not even consider the full power $\w_N^s$.
We furthermore need to examine under which conditions
\[
D + 4 N_6 r N' \text { is a square modulo $4 R N_6$ and }
D \equiv D_0 \pmod {12}.
\]
Concerning the second to last line, for instance, the condition
becomes
\[
D + 12 \, \frac {N}{3}  \text { is a square modulo $36$ and }
D \equiv 9 \pmod {12}.
\]
Thus,
$D + 12 \, \frac {N}{3} \equiv 9 \pmod {36}$, and
depending on $\frac {N}{3} \bmod 3$, only one value of
$D \pmod {36}$ remains.

For $N= 3$, for instance, or more generally
$\frac {N}{3} \equiv 1 \pmod 3$, we obtain the following
class invariants.

\begin{center}
\begin{tabular}{|l|l|r|}\hline
$B$ & $D\bmod 36$ &  $e$ \\\hline
0:1 & $0, 12$     & $12$ \\
0:1 & $9, 21$     &  $6$ \\
1:3 & $24$        &  $4$ \\
2:3 & $4, 16, 28$ &  $4$ \\
1:3 & $33$        &  $2$ \\
2:3 & $1, 13, 25$ &  $2$ \\
\hline
\end{tabular}
\end{center}
To illustrate this, we give the following table of examples:
$$\begin{array}{|r|c|r|l|}\hline
N & f & -D & H_{D}[f] \\ \hline
3 & \w_{3}^{12} & 24 &
{X}^{2}-162\,X+729
\\
3 & \w_{3}^{6} & 15 &
{X}^{2}-3\, \left( 2\,\omega-1 \right) X-27
\\
3 & \w_{3}^{4} & 12 &
    X-3
\\
3 & \w_{3}^{4} & 8 &
    X-1-2\,\omega
\\
3 & \w_{3}^{2} & 3 &
    X-\omega-1
\\
3 & \w_{3}^{2} & 11 &
    X-\omega
\\
\hline
\end{array}$$
\subsubsection{The case \texorpdfstring {$N\equiv 9\bmod 12$}{N=9 mod 12}}

We have $s = 3$ for squares in that family (for instance, $N = 3^{2n}$)
and may then reach $\w_N$. Otherwise, $s= 6$, and the only possible
smaller exponent is $2$.

\begin{center}
\begin{tabular}{|r|r||l|l||r||l|}\hline
$N$ & $s$
  & $B$
  & $D$
  & $e$ & propositions(s) \\ \hline
$9 \bmod 12$, $\neq\Box$ & 6 & 1:3 & $0\bmod 3$ & 2 & PROP30a \\
$9 \bmod 12$, $\neq\Box$ & 6 & 2:3 & $1\bmod 3$ & 2 & PROP30b \\
$9 \bmod 12$, $=\Box$ & 3 & 1:3 & $0\bmod 3$ & 1 & PROP30a    \\
$9 \bmod 12$, $=\Box$ & 3 & 2:3 & $1\bmod 3$ & 1 & PROP30b    \\
\hline
\end{tabular}
\end{center}
We give two examples, one for $N=21$, the second for $N=9$. For the
former, we find
\begin{center}
\begin{tabular}{|l|l|r|}\hline
$B$ & $D\bmod 252$ & $e$ \\
\hline
0:1& 0, 9, 21, 36, 57, 72, 81, 84, 93, 120, 144, 156, 165, 189, 225,
228 & 6 \\
1:3& 60, 105, 141, 168, 177, 204, 240, 249 & 2 \\
2:3& 1, 4, 16, 25, 28, 37, 49, 64, 85, 88, 100, 109, 112, 121, 133, 148,
& 2 \\
   & 169, 172, 184, 193, 196, 205, 217, 232  & \\
\hline
\end{tabular}
\end{center}
$$\begin{array}{|r|c|r|l|}\hline
N & f & -D & H_{D}[f] \\ \hline
21 & \w_{21}^{6} & 24 &
    {X}^{2}+ \left( 108+102\,\omega \right) X-6345+2754\,\omega
\\
21 & \w_{21}^{2} & 3 &
    X+\omega+4
\\
21 & \w_{21}^{2} & 20 &
    {X}^{2}+ \left( -2+4\,\omega \right) X-19-4\,\omega
\\
\hline
\end{array}$$
For $N = 9$, we get:
\begin{center}
\begin{tabular}{|l|l|r|}\hline
$B$ & $D\bmod 108$ & $e$ \\
\hline
0:1& 9, 36 & 3 \\
1:3& 0, 45, 72, 81 & 1 \\
2:3& 1, 4, 13, 16, 25, 28, 37, 40, 49, 52, & 1 \\
   & 61, 64, 73, 76, 85, 88, 97, 100 & \\
\hline
\end{tabular}
\end{center}
$$\begin{array}{|r|c|r|l|}\hline
N & f & -D & H_{D}[f] \\ \hline
9 & \w_{9}^{3} & 72 &
    {X}^{2}-18\,X+27
\\
9 & \w_{9} & 27 &
    X-\omega-1
\\
9 & \w_{9} & 8 &
    X-1-\omega
\\
\hline
\end{array}$$

\subsection{The case \texorpdfstring {$N$}{N} even}

A look at \S\ref{sct:lowerbis} immediately shows the complexity
of the results when $N$ is even. We distinguish the cases
$\lambda = 1$ (in which $N$ cannot be a square) and
$\lambda \geq 2$ with $N$ a square or not.

\subsubsection{The case \texorpdfstring {$\lambda = 1$}{lambda=1}}
\label{sssct:lambda1}

Three values are concerned, namely $N\bmod 12 \in \{2, 6, 10\}$. We
have $s = 24$ for $N \bmod 12 \in \{ 2, 6 \}$, whereas $s = 8$
for $N \equiv 10 \pmod {12}$.

\begin{center}
\small
\begin{tabular}{|@{}r|r@{\,}||l|l@{\,}||r@{\,}||l@{}|}\hline
\tiny $N\bmod12$ &  $s$ & $B$          & $D$                            &  $e$ & proposition(s) \\
\hline
 $2$       & $24$ & 1:2          & ---                            & $12$ & PROP20a \\
 $2$       & $24$ & 0:2          & $1\bmod8$                      & $12$ & PROP20b \\
 $2$       & $24$ & 1:3          & $1 \bmod 3$                    &  $8$ & PROP32 \\
 $2$       & $24$ & 2:3          & $2 \bmod 3$                    &  $8$ & PROP32 \\
 $2$       & $24$ & 1:4          & $1$, $4\bmod8$; $0\bmod16$     &  $6$ & PROP44  \\
 $2$       & $24$ & 3:4          & $1\bmod8$; $8\bmod16$          &  $6$ & PROP412ab \\
 $2$       & $24$ & 1:2$\cap$1:3 & $1 \bmod 3$                    &  $4$ & PROP20a+PROP32\\
 $2$       & $24$ & 1:2$\cap$2:3 & $2 \bmod 3$                    &  $4$ & PROP20a+PROP32\\
 $2$       & $24$ & 0:2$\cap$1:3 & $1 \bmod 24$                   &  $4$ & PROP20b+PROP32\\
 $2$       & $24$ & 0:2$\cap$2:3 & $17 \bmod 24$                  &  $4$ & PROP20b+PROP32\\
 $2$       & $24$ & 1:4$\cap$1:3 & $1$, $4\bmod24$; $16\bmod48$   &  $2$ & PROP44+PROP32\\
 $2$       & $24$ & 1:4$\cap$2:3 & $17$, $20\bmod24$; $32\bmod48$ &  $2$ & PROP44+PROP32\\
 $2$       & $24$ & 3:4$\cap$1:3 & $1\bmod24$; $40\bmod48$        &  $2$ & PROP412ab+PROP32\\
 $2$       & $24$ & 3:4$\cap$2:3 & $17\bmod24$; $8\bmod48$        &  $2$ & PROP412ab+PROP32\\
\hline
 $6$       & $24$ & 1:2          & ---                            & $12$ & PROP20a \\
 $6$       & $24$ & 0:2          & $1\bmod8$                      & $12$ & PROP20b \\
 $6$       & $24$ & 1:3          & $0\bmod3$                      &  $8$ & PROP30a\\
 $6$       & $24$ & 2:3          & $1\bmod3$                      &  $8$ & PROP30b\\
 $6$       & $24$ & 1:4          & $1$, $4\bmod8$; $0\bmod16$     &  $6$ & PROP44  \\
 $6$       & $24$ & 3:4          & $1\bmod8$; $8\bmod16$          &  $6$ & PROP412ab \\
 $6$       & $24$ & 1:2$\cap$1:3 & $0\bmod3$                      &  $4$ & PROP20a+PROP30a\\
 $6$       & $24$ & 1:2$\cap$2:3 & $1\bmod3$                      &  $4$ & PROP20a+PROP30b\\
 $6$       & $24$ & 0:2$\cap$1:3 & $9 \bmod 24$                   &  $4$ & PROP20b+PROP30a\\
 $6$       & $24$ & 0:2$\cap$2:3 & $1 \bmod 24$                   &  $4$ & PROP20b+PROP30b\\
 $6$       & $24$ & 1:4$\cap$1:3 & $9$, $12\bmod24$; $0\bmod48$   &  $2$ & PROP44+PROP30a\\
 $6$       & $24$ & 1:4$\cap$2:3 & $1$, $4\bmod24$; $16\bmod48$   &  $2$ & PROP44+PROP30b\\
 $6$       & $24$ & 3:4$\cap$1:3 & $9\bmod24$; $24\bmod48$        &  $2$ & PROP412ab+PROP30a\\
 $6$       & $24$ & 3:4$\cap$2:3 & $1\bmod24$; $40\bmod48$        &  $2$ & PROP412ab+PROP30b\\
\hline
$10$       &  $8$ & 1:2          & ---                            &  $4$ & PROP20a \\
$10$       &  $8$ & 0:2          & $1 \bmod 8$                    &  $4$ & PROP20b \\
$10$       &  $8$ & 1:4          & $1$, $4\bmod8$; $0\bmod16$     &  $2$ & PROP44  \\
$10$       &  $8$ & 3:4          & $1\bmod8$; $8\bmod16$          &  $2$ & PROP412ab \\
\hline
\end{tabular}
\end{center}

The case $N=2$ corresponds to Weber's classical functions. We present the
case $N=6$ in more detail, illustrating the complexity of the process.

\begin{center}
\begin{tabular}{|l|l|r|}\hline
$B$ & $D\bmod 288$ & $e$ \\\hline
0:1& 0, 36, 96, 132, 144, 180, 240, 276 & 24 \\
1:2& 60, 252 & 12 \\
1:3& 48, 84, 192, 228 & 8 \\
2:3& 4, 16, 52, 64, 100, 112, 148, 160, 196, 208, 244, 256 & 8 \\
3:4& 24, 72, 168, 216 & 6 \\
1:4& 9, 33, 81, 105, 153, 177, 225, 249 & 6 \\
1:4& 108, 204 & 6 \\
1:2 $\cap$ 1:3 & 156 & 4 \\
1:2 $\cap$ 2:3 & 28, 124, 220 & 4 \\
3:4 $\cap$ 1:3 & 120, 264 & 2 \\
1:4 $\cap$ 1:3 & 57, 129, 201, 273 & 2 \\
1:4 $\cap$ 1:3 & 12 & 2 \\
3:4 $\cap$ 2:3 & 40, 88, 136, 184, 232, 280 & 2 \\
1:4 $\cap$ 2:3 & 1, 25, 49, 73, 97, 121, 145, 169, 193, 217, 241, 265 & 2 \\
1:4 $\cap$ 2:3 & 76, 172, 268 & 2 \\
\hline
\end{tabular}
\end{center}

$$\begin{array}{|r|c|r|l|}\hline
N & f & -D & H_{D}[f] \\ \hline
6 & \w_{6}^{24} & 12 &
X+186624
\\
6 & \w_{6}^{12} & 36 &
{X}^{2}-3888\,\omega\,X+1259712
\\
6 & \w_{6}^{8} & 60 &
{X}^{2}+ \left( 432\,\omega-720 \right) X+20736
\\
6 & \w_{6}^{8} & 32 &
{X}^{2}+ \left( 112+64\,\omega \right) X-1088-3584\,\omega
\\
6 & \w_{6}^{6} & 72 &
{X}^{2}-216\,X-5832
\\
6 & \w_{6}^{6} & 39 &
{X}^{4}+ \left( 3\,\omega-42 \right) {X}^{3}+ \left( 486\,\omega+108
 \right) {X}^{2} \\
&&&+ \left( -648\,\omega+9072 \right) X+6561\,\omega- 45198
\\
6 & \w_{6}^{6} & 84 &
{X}^{4}+ \left( 324+60\,\omega \right) {X}^{3}+14688\,{X}^{2}\\
&&&+ \left( 69984-12960\,\omega \right) X+46656
\\
6 & \w_{6}^{4} & 132 &
{X}^{4}+ \left( 144-12\,\omega \right) {X}^{3}+2196\,{X}^{2}\\
&&&+ \left( 5184+432\,\omega \right) X+1296
\\
6 & \w_{6}^{4} & 68 &
{X}^{4}+ \left( -32+4\,\omega \right) {X}^{3}+ \left( -204-96\,\omega
 \right) {X}^{2}\\
&&&+ \left( 1152-144\,\omega \right) X-752+256\,\omega
\\
6 & \w_{6}^{2} & 24 &
{X}^{2}-\omega\,X-6
\\
6 & \w_{6}^{2} & 15 &
{X}^{2}+ \left( -2\,\omega-2 \right) X+3\,\omega-3
\\
6 & \w_{6}^{2} & 276 &
{X}^{8}+ \left( -12-4\,\omega \right) {X}^{7}+ \left( 132+6\,\omega
 \right) {X}^{6}\\
&&&-144\,{X}^{5}-576\,{X}^{4}-864\,{X}^{3}+ \left( 4752-
216\,\omega \right) {X}^{2}\\
&&&+ \left( -2592+864\,\omega \right) X+1296
\\
6 & \w_{6}^{2} & 8 &
X+2+\omega
\\
6 & \w_{6}^{2} & 23 &
{X}^{3}-6\,{X}^{2}+ \left( -\omega+15 \right) X+\omega-15
\\
6 & \w_{6}^{2} & 20 &
{X}^{2}+ \left( 2-2\,\omega \right) X-4-2\,\omega
\\
\hline
\end{array}$$

\subsubsection{The case \texorpdfstring {$\lambda \geq 2$}{lambda>=2}}

We have to study three values of $N\bmod 12$, namely, $0$, $4$ and $8$,
for which $s = 24$, $8$, and $24$, respectively. The cases $N \equiv 0$ or
$4$ authorise
squares, so that the results become somewhat lengthy.

When $N\equiv 4\bmod 12$, we find
\begin{center}
\begin{tabular}{|r|r||l|l||r||l|}\hline
$N$ &  $s$ & $B$          & $D$                            &
$e$ & proposition(s) \\
\hline
$4\bmod12$ & $8$ & 1:2 & --- & $4$ & PROP20a\\
$4\bmod12$ & $8$ & 1:2 & $1\bmod 8$ & $4$ & PROP20b\\
$4\bmod12$ & $8$ & 1:4 & $1\bmod 8$ & $2$ & PROP44a\\
$4\bmod12$ & $8$ & 1:4 & $0\bmod 16$ & $2$ & PROP44b\\
$4\bmod12$ & $8$ & 3:4 & $1\bmod 8$ & $2$ & PROP412a\\
$4\bmod12$ & $8$ & 3:4 & $4\bmod 8$ & $2$ & PROP412c\\
\hline
$4\bmod12$, $= \Box$ & $8$ & 3:8 & $1\bmod 8$ & $1$ & PROP8a\\
$4\bmod12$, $= \Box$ & $8$ & 7:8 & $1\bmod 8$ & $1$ & PROP8a\\
$4\bmod12$, $= \Box$ & $8$ & 1:8 & $0\bmod 32$ & $1$ & PROP8b\\
$4\bmod12$, $= \Box$ & $8$ & 5:8 & $16\bmod 32$ & $1$ & PROP8c\\
\hline
\end{tabular}
\end{center}

When $N\equiv 8\bmod 12$, it cannot be a square, and the results are:
\begin{center}
\begin{tabular}{|r|r||l|l||r||l|}\hline
$N\bmod12$ &  $s$ & $B$          & $D$                            &  $e$ & proposition(s) \\
\hline
$8$ & $24$ & 1:2 & --- & $12$ & PROP20a\\
$8$ & $24$ & 1:2 & $1\bmod 8$ & $12$ & PROP20b\\
$8$ & $24$ & 1:4 & $1\bmod 8$ & $6$ & PROP44a\\
$8$ & $24$ & 1:4 & $0\bmod 16$ & $6$ & PROP44b\\
$8$ & $24$ & 3:4 & $1\bmod 8$ & $6$ & PROP412a\\
$8$ & $24$ & 3:4 & $4\bmod 8$ & $6$ & PROP412c\\
$8$ & $24$ & 1:3 & $1\bmod 3$ & $8$ & PROP32\\
$8$ & $24$ & 2:3 & $2\bmod 3$ & $8$ & PROP32\\
$8$ & $24$ & 1:2 $\cap$ 1:3 & $1\bmod 3$ & $4$ & PROP20a+PROP32\\
$8$ & $24$ & 1:2 $\cap$ 2:3 & $2\bmod 3$ & $4$ & PROP20a+PROP32\\
$8$ & $24$ & 1:2 $\cap$ 1:3 & $1\bmod 24$ & $4$ & PROP20b+PROP32\\
$8$ & $24$ & 1:2 $\cap$ 2:3 & $17\bmod 24$ & $4$ & PROP20b+PROP32\\
$8$ & $24$ & 1:4 $\cap$ 1:3 & $1\bmod 24$ & $2$ & PROP44a+PROP32\\
$8$ & $24$ & 1:4 $\cap$ 2:3 & $17\bmod 24$ & $2$ & PROP44a+PROP32\\
$8$ & $24$ & 1:4 $\cap$ 1:3 & $16\bmod 48$ & $2$ & PROP44b+PROP32\\
$8$ & $24$ & 1:4 $\cap$ 2:3 & $32\bmod 48$ & $2$ & PROP44b+PROP32\\
$8$ & $24$ & 3:4 $\cap$ 1:3 & $1\bmod 24$ & $2$ & PROP412a+PROP32\\
$8$ & $24$ & 3:4 $\cap$ 2:3 & $17\bmod 24$ & $2$ & PROP412a+PROP32\\
$8$ & $24$ & 3:4 $\cap$ 1:3 & $4\bmod 24$ & $2$ & PROP412c+PROP32\\
$8$ & $24$ & 3:4 $\cap$ 2:3 & $20\bmod 24$ & $2$ & PROP412c+PROP32\\
\hline
\end{tabular}
\end{center}
Finally, for $N\equiv 0\bmod 12$, we obtain the following results:
\begin{center}
\begin{tabular}{|r|r||l|l||r||l|}\hline
$N$ &  $s$ & $B$          & $D$                            &  $e$ & proposition(s) \\
\hline
$12$ & $24$ & 1:2 & --- & $12$ & PROP20a\\
$12$ & $24$ & 1:2 & $1\bmod 8$ & $12$ & PROP20b\\
$12$ & $24$ & 1:4 & $1\bmod 8$ & $6$ & PROP44a\\
$12$ & $24$ & 1:4 & $0\bmod 16$ & $6$ & PROP44b\\
$12$ & $24$ & 3:4 & $1\bmod 8$ & $6$ & PROP412a\\
$12$ & $24$ & 3:4 & $4\bmod 8$ & $6$ & PROP412c\\
$12$ & $24$ & 1:3 & $0\bmod 3$ & $8$ & PROP30a\\
$12$ & $24$ & 2:3 & $1\bmod 3$ & $8$ & PROP30b\\
$12$ & $24$ & 1:2 $\cap$ 1:3 & $0\bmod 3$ & $4$ & PROP20a+PROP30a\\
$12$ & $24$ & 1:2 $\cap$ 2:3 & $1\bmod 3$ & $4$ & PROP20a+PROP30b\\
$12$ & $24$ & 1:2 $\cap$ 1:3 & $9\bmod 24$ & $4$ & PROP20b+PROP30a\\
$12$ & $24$ & 1:2 $\cap$ 2:3 & $1\bmod 24$ & $4$ & PROP20b+PROP30b\\
$12$ & $24$ & 1:4 $\cap$ 1:3 & $9\bmod 24$ & $2$ & PROP44a+PROP30a\\
$12$ & $24$ & 1:4 $\cap$ 2:3 & $1\bmod 24$ & $2$ & PROP44a+PROP30b\\
$12$ & $24$ & 1:4 $\cap$ 1:3 & $0\bmod 48$ & $2$ & PROP44b+PROP30a\\
$12$ & $24$ & 1:4 $\cap$ 2:3 & $16\bmod 48$ & $2$ & PROP44b+PROP30b\\
$12$ & $24$ & 3:4 $\cap$ 1:3 & $9\bmod 24$ & $2$ & PROP412a+PROP30a\\
$12$ & $24$ & 3:4 $\cap$ 2:3 & $1\bmod 24$ & $2$ & PROP412a+PROP30b\\
$12$ & $24$ & 3:4 $\cap$ 1:3 & $12\bmod 24$ & $2$ & PROP412c+PROP30a\\
$12$ & $24$ & 3:4 $\cap$ 2:3 & $4\bmod 24$ & $2$ & PROP412c+PROP30b\\
\hline
$12$ & $24$ & 3:8 & $1\bmod 8$ & $3$ & PROP8a\\
$12$ & $24$ & 7:8 & $1\bmod 8$ & $3$ & PROP8a\\
$12$ & $24$ & 1:8 & $0\bmod 32$ & $3$ & PROP8b\\
$12$ & $24$ & 5:8 & $16\bmod 32$ & $3$ & PROP8c\\
$12$ & $24$ & 3:8 $\cap$ 1:3 & $9\bmod 24$ & $1$ & PROP8a+PROP30a\\
$12$ & $24$ & 3:8 $\cap$ 2:3 & $1\bmod 24$ & $1$ & PROP8a+PROP30b\\
$12$ & $24$ & 7:8 $\cap$ 1:3 & $9\bmod 24$ & $1$ & PROP8a+PROP30a\\
$12$ & $24$ & 7:8 $\cap$ 2:3 & $1\bmod 24$ & $1$ & PROP8a+PROP30b\\
$12$ & $24$ & 1:8 $\cap$ 1:3 & $0\bmod 96$ & $1$ & PROP8b+PROP30a\\
$12$ & $24$ & 1:8 $\cap$ 2:3 & $64\bmod 96$ & $1$ & PROP8b+PROP30b\\
$12$ & $24$ & 5:8 $\cap$ 1:3 & $48\bmod 96$ & $1$ & PROP8c+PROP30a\\
$12$ & $24$ & 5:8 $\cap$ 2:3 & $16\bmod 96$ & $1$ & PROP8c+PROP30b\\
\hline
\end{tabular}
\end{center}

For $N=4$, these results translate as follows:

\begin{center}
\begin{tabular}{|l|l|r|}\hline
$B$ & $D\bmod 128$ & $e$ \\\hline
0:1& $\equiv 4 \pmod {32}$ & 8 \\
1:2& 16, 32, 80, 96 & 4 \\
3:4& $\equiv 20 \pmod {32}$ & 2 \\
1:4& 64 & 2 \\
3:8& $\equiv 1 \pmod 8$  & 1 \\
1:8& 0 & 1 \\
5:8& $\equiv 48 \pmod {64}$ & 1 \\
\hline
\end{tabular}
\end{center}

$$\begin{array}{|r|c|r|l|}\hline
N & f & -D & H_{D}[f] \\ \hline
4 & \w_{4}^{8} & 28 &
X-48\,\omega+32
\\
4 & \w_{4}^{4} & 32 &
{X}^{2}-8\,\omega\,X-16
\\
4 & \w_{4}^{2} & 12 &
X-2\,\omega
\\
4 & \w_{4}^{2} & 64 &
{X}^{2}+ \left( -4-4\,\omega \right) X+4\,\omega
\\
4 & \w_{4} & 7 &
X-\omega
\\
4 & \w_{4} & 128 &
{X}^{4}+ \left( -4-2\,\omega \right) {X}^{3}+6\,\omega\,{X}^{2}+
 \left( 8-4\,\omega \right) X-4
\\
4 & \w_{4} & 16 &
X-1-\omega
\\
\hline
\end{array}$$

The precise results for $N = 16$ are the following:

\begin{center}
\begin{tabular}{|l|l|r|l|}\hline
$B$ & $D\bmod 512$ & $e$ \\\hline
0:1& $\equiv 16 \pmod {128}$ & 8 \\
1:2& 64, 128, 320, 384 & 4 \\
3:4& $\equiv 4 \pmod {32}$ & 2 \\
1:4& 256 & 2 \\
3:8& $\equiv 1\pmod 8$ & 1 \\
1:8& 0, 192, 448 & 1 \\
5:8& $\equiv 80 \pmod {128}$ & 1 \\
\hline
\end{tabular}
\end{center}

$$\begin{array}{|r|c|r|l|}\hline
N & f & -D & H_{D}[f] \\ \hline
16 & \w_{16}^{8} & 112 &
{X}^{2}+ \left( 12288\,\omega-8192 \right) X-196608\,\omega-917504
\\
16 & \w_{16}^{4} & 128 &
{X}^{4}+ \left( 128+192\,\omega \right)
{X}^{3}+6656\,\omega\,{X}^{2}\\
&&&+ \left( -32768+49152\,\omega \right) X-65536
\\
16 & \w_{16}^{2} & 28 &
X+2\,\omega-4
\\
16 & \w_{16}^{2} & 256 &
{X}^{4}+ \left( 16-48\,\omega \right) {X}^{3}+ \left( -288+288\,\omega
 \right) {X}^{2}\\
&&&+ \left( 768-256\,\omega \right) X-256\,\omega
\\
16 & \w_{16} & 7 &
X-\omega-1
\\
16 & \w_{16} & 64 &
{X}^{2}-4\,X+4
\\
16 & \w_{16} & 48 &
{X}^{2}+4\,X+4
\\
\hline
\end{array}$$

\subsection {Reality of class polynomials}

The argumentation of the proof of Theorem~\ref {th:reality} carries over to the
lower powers of $\w_N$ and shows that the class polynomial is real
whenever for some form $[A, B, C]$ in the $\frac {s}{e} N$-system the inverse
form $[A, -B, C]$ satisfies the congruence constraints of the system as well.
This is precisely the case when $B$ is divisible by $\frac
{s}{e}N$. In particular, this implies that $N \mid D$, and inspection
of the previous results proves the following theorem.

\begin {theorem}
\label {th:lowreality}
Under the general assumptions of \S\ref {sct:specialization},
the characteristic polynomial of $\w_N^e (\alpha)$ is real
whenever $N \mid D$ and $\frac {s}{e} N \mid B$.
For $e < s$, this is possible only in the following cases:
\begin {alphenumerate}
\item
$N$ odd:
\begin{center}
\begin{tabular}{|r|r||l|l||r|}\hline
$N$                    &  $s$ & $B$ & $D$             & $e$ \\
\hline
$ 5 \bmod 12$          &  $6$ & $1$:$3$ & $1\bmod 3$  & $2$ \\
$ 5 \bmod 12$          &  $6$ & $2$:$3$ & $2\bmod 3$  & $2$ \\
\hline
$11 \bmod 12$          & $12$ & $1$:$3$ & $1\bmod 3$  & $4$ \\
$11 \bmod 12$          & $12$ & $2$:$3$ & $2\bmod 3$  & $4$ \\
\hline
$ 3 \bmod 12$          & $12$ & $1$:$3$ & $6 \bmod 9$ & $4$ \\
\hline
$9\bmod12$, $\neq\Box$ &  $6$ & $1$:$3$ & $18 \bmod 27$ & $2$ \\
$9\bmod12$, $= \Box$   &  $3$ & $1$:$3$ & $18 \bmod 27$ & $1$ \\
\hline
\end {tabular}
\end {center}
\item
$2 \mid\mid N$ and $4 \mid D$
\begin {alphenumerate}
\item
$\frac {s}{e}$ is even and $8 \mid\mid D$
\item
$\frac {s}{e} = 3$
\end {alphenumerate}
\item
$4 \mid N$ and $16 \mid D$
\end {alphenumerate}
\end {theorem}

\begin{proof}
We again start from $B^2 \equiv D + 4 r N \pmod {4 R N}$, where
in fact $R = \frac {s}{e}$ is a non-trivial divisor of $24$. Then
the hypotheses of the theorem translate as $B = N R B'$ and
$D = N D'$, so that
\begin{equation}\label{eq:real}
N R^2 {B'}^2\equiv D' + 4 r \pmod {4 R}.
\end{equation}
This immediately implies
\begin {align}
\label{eq:real3}
& D' \equiv -r \pmod 3 & \text { if } 3 \mid R \\
\label{eq:real4}
& 4 \mid D'            & \text { if } 2 \mid R
\end {align}

\begin {alphenumerate}
\item
The assertions are a direct consequence of \eqref{eq:real3} and
\eqref{eq:real4}, together with the tables in \S\ref{sct:Nodd}.

\item
If $N$ is even, from $N \mid D$ we immediately have $4 \mid D$.

If $R$ is even, then moreover \eqref{eq:real4} yields that $8 \mid D$.
Going through the table in \S\ref{sssct:lambda1} shows that then $r$
is odd, and \eqref{eq:real} implies
that $D' \equiv -4r \equiv 4 \pmod 8$ and $8 \mid\mid D$.

\item
If $4 \mid N$, then \eqref{eq:real} shows that $4 \mid D'$, whence $16 \mid D$.
\end {alphenumerate}
\end{proof}

We end this section with related results concerning the functions
$\sqrt{D} \, \w_N^e$. Since $\sqrt {D} \in \OO$, a singular value
$\sqrt{D} \, \w_N^e (\alpha)$ is a class invariant if and only if
$\w_N^e (\alpha)$ is, and integrality of the class polynomial coefficients
carries over. In some cases, however, the additional factor $\sqrt D$
may lead to rational class polynomials.

\begin{lemma}\label{lem:v}
Let $N \not\equiv 1 \pmod 8$,
$\alpha = \frac {-B+\sqrt{D}}{2}$ and $e$ be such that
$\frac {s}{e}$ is even, $\frac {s}{2 e} N \mid B$ and
$\frac {s}{e} N \nmid B$. Then $\w_N(\alpha)^e \in i\R$.
\end{lemma}

\begin{proof}
Write $\w_N = f_0 f_1$, where $f_0 = q^{- \frac {N-1}{24 N}}$ and
$f_1$ is a power series in $q^{1/N}$.
Notice that if $N \mid B$, then
$q^{1/N} (\alpha) = e^{2 \pi i \alpha / N} \in \R$.
So $\w_N^e (\alpha)$ is real up to the factor $f_0 (\alpha)^e$,
which itself is real up to the factor
$e^{\frac {2 \pi i}{4} \cdot \frac {s (N-1)}{24}
\cdot \frac {2 e B}{s N}}$.
This is an odd power of $i$ under the hypotheses of the lemma;
$N \not\equiv 1 \pmod 8$ is needed to ensure that
$\frac {s (N - 1)}{24}$ is odd.
\end{proof}

\begin {lemma}
Let $f$ be a modular function and $\alpha \in \OO$ such that $f (\alpha)$ is
a class invariant and a real number. Then
$H_D [f] \in \Q [X]$.
\end {lemma}

\begin{proof}
This is a trivial application of Galois theory.
The complex conjugate $\overline {f (\alpha)}$ is a root of
$\overline {H_D[f]}$. Since $\overline {f (\alpha)} = f (\alpha)$,
this implies that $\overline {H_D[f]}$ is a multiple of the minimal
polynomial $H_D [f]$ of $f (\alpha)$, so both are the same, and
$H_D [f]$ has coefficients in $K \cap \R = \Q$.
\end {proof}

Combining the lemmata yields the following result.

\begin {theorem}
Under the general assumptions of \S\ref {sct:specialization},
the characteristic polynomial of $\sqrt {D} \, \w_N^e (\alpha)$ is real
whenever $N \not\equiv 1 \pmod 8$, $N \mid D$, $\frac {s}{e}$ is even,
$\frac {s}{2e} N \mid B$ and $\frac {s}{e} N \nmid B$.
\end {theorem}

For instance, we may apply this theorem to the cases $N\in \{2, 3,
4, 7\}$, in which Propositions~\ref{prop:Dodd}
or~\ref{prop:Neventheta2} hold:
$$\begin{array}{|c|c|c|c|}\hline
N & D & B & e\\ \hline
2 & 12\bmod 16& \pm 2 & 12 \\ 
  & 24 \bmod 96 & \pm 12 & 6 \\ 
\hline
3 & 9\bmod 12 & \pm 3 & 6 \\
\hline
7 & 21\bmod 28 & \pm 7 & 2 \\
\hline
4 & 0\bmod 32 & \pm 4 & 4 \\ 
\hline
\end{array}$$
As numerical examples, we find:
$$H_{-72}[\sqrt{-72}\, \w_2^6] = X^2 + 720 X + 576,$$
$$H_{-51}[\w_3^6](X) = X^2 + 6 \sqrt{-51} X - 27,$$
$$H_{-51}[\sqrt{-51} \, \w_3^6](X) = X^2 -306 X + 1377.$$

\section{Heights and comparison with other invariants}
\label{sct:heights}

Let $f$ be a modular function yielding class invariants and $\Phi[f](F, J)$ the
associated modular polynomial such that $\Phi[f] (f, j) = 0$. It is shown in
\cite{EnMo02} that asymptotically for $|D| \to \infty$, the height of the
class invariant $f(\alpha)$ is $c (f)$ times the height of $j (\alpha)$,
where
\begin {equation}
\label {eq:c}
c(f) = \frac{\mathrm{deg}_J(\Phi[f])}{\mathrm{deg}_F(\Phi[f])}
\end {equation}
depends only on $f$. It is then clear that $c (f^r) = r c(f)$ for rational $r$.
So to obtain $c (\w_N^e)$, it is sufficient to determine the degrees of the
modular polynomials of the full power $\w_N^s$, where $s$ is as defined in
Theorem~\ref {th:full}.

\subsection{Modular polynomials for \texorpdfstring {$\w_N^s$}{wNs}}

Since $\w_N^s$ is modular for $\Gamma^0 (N)$ by Theorem~\ref {th:F24N},
we have
\[
\PhiNc := \Phi [\w_N^s]
= \prod_{M \in \Gamma^0(N) \backslash \Gamma} (F - \w_N^s \circ M).
\]
So $\deg_F \PhiNc = \psi (N)
= N \prod_{p \text { prime, } p \mid N} \left( 1 + \frac {1}{p} \right)$.
The degree in $J$ is obtained by examining the $q$-developments
of the conjugates $\w_N^s \circ M$ of $\w_N^s$.

\begin{proposition}[Oesterl\'e]
The cosets of $\Gamma^0(N) \backslash \Gamma$ can be split into the
following three families:
$$T^{\nu} = \left(\begin{array}{cc} 1 & \nu \\ 0 & 1 \\ \end{array}\right), 0
\leq \nu < N,$$
$$S = \left(\begin{array}{cc} 0 & -1 \\ 1 & 0 \\ \end{array}\right),$$
$$M_{k, k'} = \left(\begin{array}{cc} k & k k'-1 \\1 & k'\\\end{array}\right)$$
with $1 < k < N$, $\gcd(k, N) > 1$ and $0 \leq k' < \mu(k)$ where
$\mu(k)$ is the smallest integer for which $\gcd(\mu(k) k -1, N)=1$.
\end{proposition}

Using \eqref {eq:w_N|M}, we find
\begin{proposition}
$$(\w_N^s \circ T) (z) = \w_N(z+\nu)^s, 0 \leq \nu < N,$$
$$(\w_N^s \circ S) (z) = \left( \sqrt{N} \, \frac {\eta(N z)}{\eta(z)} \right)^s,$$
$$(\w_N^s \circ M_{k, k'}) (z) = \left( \zeta_{k, k'} \sqrt {\delta_k} \,
\frac{\eta\left( \frac {\delta_k z +
c_{k, k'}}{N/\delta_k} \right) }{\eta(z)}\right)^s,$$
where $\delta_k = \gcd(k, N)$, $\zeta_{k, k'}$ is a 24-th root
of unity and $c_{k, k'}$ is a rational integer.
\end{proposition}

The proposition shows in particular that all conjugates of $\w_N^s$ have integral
and that $\w_N^s$ and $\w_N^s \circ S$ have rational $q$-expansions.
The $q$-expansion principle now implies that $\PhiNc \in \Z [F, J]$,
cf.~\cite [\S 3]{Deuring58}

\begin{theorem}
$$\deg_J \Phi_N^c = \frac{s}{24} (N-1 + S(N))$$
where
\begin {equation}
\label {eq:S}
S(N) = \sum_{k : 1 < k < N, 1 < \delta_k=\gcd(k, N) < \sqrt{N}}
\mu(k) \, \left( 1 - \frac {\delta_k^2}{N} \right).
\end {equation}
\end{theorem}

\begin{proof}
Consider $\PhiNc$ as a polynomial in $F$ with coefficients in $\Z [J]$.
Following the same reasoning as in \cite{EnSc05}, we see that the
coefficient of highest degree in $J$ is obtained when all conjugates are
multiplied together whose $q$-expansions have strictly negative order;
since the $q$-expansion of $j$ starts with $q^{-1}$, the degree in $J$ is
then the opposite of this order.
The $\w_N(z+\nu)^s$ have negative order $- \frac {s (N-1) }{24 N}$
and contribute a total of $- \frac {s (N-1)}{24}$.
The function $\w_N^s \circ S$ has positive order.
The conjugates coming from $M_{k, k'}$ have order
$\frac {s}{24} \left( \frac {\delta_k^2}{N} - 1 \right)$, which is
negative whenever $\delta_k < \sqrt N$.
\end{proof}

Let us note a list of useful corollaries.
\begin{proposition}
When $N = \ell^n$ for a prime $\ell$ and $n \geq 1$, then
$$S(N)= \left\{\begin{array}{ll}
(\ell^m-1)^2 & \text{ if } n = 2m+1,\\
(\ell^m-1)(\ell^{m+1}-1) & \text{ if } n = 2m+2.
\end{array}\right.$$
\end{proposition}

\begin{proof}
The $k$ occurring in \eqref {eq:S} are the $(k_1 + \ell k_2) \ell^r$
with $1 \leqslant k_1 < \ell$, $1 \leqslant r \leqslant m$
and $0 \leqslant k_2 < \ell^{n-r-1}$ (so that $k < N$);
they yield $\delta_k = \ell^r$ and $\mu (k) = 1$. Hence,
\[
S (N)
= \sum_{r=1}^m (\ell - 1) \ell^{n-r-1} \left( 1 - \ell^{2 r - n} \right)
= \left( \ell^{n-m-1} - 1 \right) \left( \ell^m - 1 \right).
\]
\end{proof}

\begin{corollary}
When $N$ is prime or the square of a prime,
then $\deg_J \PhiNc = \frac {s(N-1)}{24}$.
\end{corollary}

\begin{proposition}
When $N = p_1 p_2$ for two primes $p_2 \geq p_1$, then $S (N) =
p_2- p_1$.
\end{proposition}

\begin{proof}
The case $p_1=p_2$ is already proven. So it remains to consider
$p_1 < \sqrt{N} < p_2$, and the integers $k$ contributing
to $S (N)$ are the $\tilde k p_1$ with $1\leq \tilde k <
p_2$. Among these, only one is such that $\gcd(k-1, N)\neq 1$,
namely the $k$ with $\tilde k \equiv 1/p_1\pmod {p_2$}; for this one,
$\mu(k) = 2$. Therefore
$$S (N) = \big( (p_2-2) \cdot 1 + 1 \cdot 2 \big)
\left( 1- \frac {p_1^2}{N} \right) = p_2 - p_1.
$$
\end{proof}

With some more effort, the constant coefficient $\PhiNc(0, J)$
could be obtained as the product of all conjugates, but it is not needed
in the following.

\subsection{Heights}
\label {ssct:heights}

Knowing the degrees of the modular polynomials, we can compare class
invariants obtained from $\w_N^e$ among themselves and with others
using \eqref {eq:c}. Of special interest is the infinite family of
invariants obtained in \cite{EnSc04} from the double $\eta$-quotients
\[
\w_{p_1, p_2}^\sigma (z) = \left( \frac {\eta \left( \frac {z}{p_1} \right)
\eta \left( \frac {z}{p_2} \right)}{\eta \left( \frac {z}{p_1 p_2} \right)
\eta (z)} \right)^\sigma,
\]
where $p_1$, $p_2$ are (not necessarily distinct) primes and
$\sigma = \frac {24}{\gcd (24, (p_1 - 1)(p_2 - 1))}$.
These functions yield class invariants whenever
$\legendre {D}{p_1} = \legendre {D}{p_2} = 1$, and in some cases
when $\legendre {D}{p_1} = 0$ or $\legendre {D}{p_2} = 0$,
see \cite[Cor.~3.1]{EnSc04}. The degrees of their modular
polynomials have been worked out in \cite[Th.~9]{EnSc05}, and we
summarise the results in the following table, in which $\ell$
and $p_1 \neq p_2$ are supposed to be prime numbers.

{
\renewcommand {\arraystretch}{1.4}
$$\begin{array}{|c|c|c|}\hline
f & c(f) & \deg_J \PhiNc \\ \hline
\w_{\ell}^e      & \frac {e (\ell-1)}{24 (\ell+1)} & \frac{s
(\ell-1)}{24} \\
\w_{\ell^2}^e    & \frac {e (\ell-1)}{24 \ell} & \frac{\ell^2-1}{24}
\text{ if } \ell > 3\\
\w_{p_1 p_2}^e   & \frac {e (p_2-1)}{24 (p_2 + 1)} & \frac{s
(p_2-1)(p_1+1)}{24} \\
\w_N^e           & \frac {e (N-1+S(N))}{24 \psi (N)} & \frac{s
(N-1+S(N))}{24}\\
\hline
\w_{\ell, \ell}^e & \frac {e (\ell-1)^2}{12 \ell (\ell + 1)} &
\frac{\sigma (\ell-1)^2}{12} \\
\w_{p_1, p_2}^e   & \frac {e (p_1-1)(p_2-1)}{12 (p_1 + 1)(p_2 + 1)} &
\frac{\sigma (p_1-1) (p_2-1)}{12}\\
\hline
\end{array}$$
}

Notice that asymptotically for $\ell$ or $p_1$, $p_2 \to \infty$,
the factors $c (f)$ tend to $\frac {e/2}{12}$ for $\w_\ell^e$ (here,
$e$ is necessarily even), $\frac {e}{12}$ for the double $\eta$ quotients
and $\frac {e}{24}$ for $\w_{\ell^2}^e$. For any discriminant $D$, there
are suitable choices of primes in arithmetic progressions modulo~$D$
such that $e/2 = 1$ resp. $e = 1$ are reachable, and $c (f)$ may become
arbitrarily close to $\frac {1}{12}$ resp. $\frac {1}{24}$. However,
at the same time, the degrees of $\PhiNc$ in $F$ and $J$ tend to
infinity, which may be undesirable in complex multiplication
applications where the modular polynomial needs to be factored
over a finite field.

In Table~\ref {tab:comparison}, we list in decreasing order of
attractiveness the
functions $f$ together with the factors $1 / c (f)$ they allow
to gain in height compared to $j$ and with the degree of the modular
polynomial in $J$, thus completing the tables of \cite{EnMo02} and
\cite[p.~21]{Enge07}. We limit ourselves to functions gaining a factor
of at least~$13$ and with degree in $J$ at most~$20$.
The function $\w_2$ is in fact the Weber function $\wfone$,
and leads to the same height as the other two Weber functions
$\wf$ and $\wftwo$.
Notice that, as indicated by the explicit formul{\ae}, transformation
levels divisible by~$2$ or~$3$ (or, in general, small primes)
tend to yield smaller class invariants.

\newcommand {\entree}[2]{\genfrac{}{}{0pt}{1}{#1}{#2}}
\begin {table}[hbt]
\label {tab:comparison}
\caption {Comparison of class invariants: height factor and degree in $J$}
\renewcommand {\arraystretch}{0.7}
\renewcommand {\arraycolsep}{0.3em}
$$
\begin{array}{cccccccccccc}
& \entree{\w_{2}}{72, 1}
&>& \entree{\w_{4}}{48, 1}
&>& \entree{\w_{2, 73}}{37, 6}
&>& \entree{\w_{2, 97}}{147/4, 8}
&>& \entree{\w_{9}}{36, 1}
&=& \entree{\w_{2}^{2}}{36, 1}
\\
\\
>& \entree{\w_{16}}{32, 6}
&>& \entree{\w_{25}}{30, 1}
&>& \entree{\w_{3, 13}}{28, 2}
&=& \entree{\w_{49}}{28, 2}
&>& \entree{\w_{81}}{27, 12}
&>& \entree{\w_{11^{2}}}{132/5, 5}
\\
\\
>& \entree{\w_{13^{2}}}{26, 7}
&>& \entree{\w_{17^{2}}}{51/2, 12}
&>& \entree{\w_{3, 37}}{76/3, 6}
&=& \entree{\w_{19^{2}}}{76/3, 15}
&>& \entree{\w_{3, 61}}{124/5, 10}
&>& \entree{\w_{5, 7}}{24, 2}
\\
\\
=& \entree{\w_{2}^{3}}{24, 1}
&=& \entree{\w_{6}^{2}}{24, 6}
&=& \entree{\w_{4}^{2}}{24, 1}
&=& \entree{\w_{3}^{2}}{24, 1}
&>& \entree{\w_{5, 13}}{21, 4}
&=& \entree{\w_{2, 13}^{2}}{21, 2}
\\
\\
>& \entree{\w_{12}^{2}}{144/7, 14}
&>& \entree{\w_{5, 19}}{20, 6}
&>& \entree{\w_{5, 31}}{96/5, 10}
&>& \entree{\w_{5, 37}}{19, 12}
&=& \entree{\w_{2, 37}^{2}}{19, 6}
&>& \entree{\w_{7, 13}}{56/3, 6}
\\
\\
>& \entree{\w_{2, 61}^{2}}{93/5, 10}
&>& \entree{\w_{7, 17}}{18, 8}
&=& \entree{\w_{15}^{2}}{18, 8}
&=& \entree{\w_{8}^{2}}{18, 8}
&=& \entree{\w_{2}^{4}}{18, 1}
&=& \entree{\w_{5}^{2}}{18, 1}
\\
\\
=& \entree{\w_{10}^{2}}{18, 4}
&>& \entree{\w_{11, 13}}{84/5, 10}
&>& \entree{\w_{3, 7}^{2}}{16, 2}
&=& \entree{\w_{35}^{2}}{16, 18}
&=& \entree{\w_{21}^{2}}{16, 6}
&=& \entree{\w_{40}^{2}}{16, 18}
\\
\\
=& \entree{\w_{14}^{2}}{16, 18}
&=& \entree{\w_{16}^{2}}{16, 6}
&=& \entree{\w_{28}^{2}}{16, 12}
&=& \entree{\w_{7}^{2}}{16, 1}
&=& \entree{\w_{3}^{3}}{16, 1}
&=& \entree{\w_{6}^{3}}{16, 6}
\\
\\
>& \entree{\w_{45}^{2}}{108/7, 14}
&>& \entree{\w_{13, 13}}{91/6, 12}
&>& \entree{\w_{55}^{2}}{72/5, 10}
&=& \entree{\w_{77}^{2}}{72/5, 20}
&=& \entree{\w_{22}^{2}}{72/5, 10}
&=& \entree{\w_{11}^{2}}{72/5, 5}
\\
\\
=& \entree{\w_{33}^{2}}{72/5, 10}
&=& \entree{\w_{27}^{2}}{72/5, 15}
&>& \entree{\w_{91}^{2}}{14, 16}
&=& \entree{\w_{65}^{2}}{14, 18}
&=& \entree{\w_{13}^{2}}{14, 1}
&>& \entree{\w_{12}^{3}}{96/7, 14}
\\
\\
>& \entree{\w_{2, 17}^{3}}{27/2, 4}
&=& \entree{\w_{85}^{2}}{27/2, 8}
&=& \entree{\w_{34}^{2}}{27/2, 16}
&=& \entree{\w_{17}^{2}}{27/2, 4}
&>& \entree{\w_{3, 19}^{2}}{40/3, 6}
&=& \entree{\w_{7\cdot 19}^{2}}{40/3, 12}
\\
\\
=& \entree{\w_{57}^{2}}{40/3, 18}
&=& \entree{\w_{19}^{2}}{40/3, 3}
&>& \entree{\w_{23}^{2}}{144/11, 11}
\end{array}
$$
\end {table}

\section{Outlook}

The presented results concern singular values of powers of $\w_N$
as class invariants.
It is possible to obtain smaller invariants by
authorising $24$-th roots of unity to enter the game. This was already
done by Weber for $N=2$ (the classical $f$-functions) and
by Gee in \cite{Gee01} for $N=3$.
For instance, $\zeta_4\w_7^2$ is
an invariant for $D=-40$, leading to the minimal polynomial
$${X}^{2}+ \left( -5+2\omega \right) X+3-4\,\omega.$$

Similarly, when $N$ is not a square and $e$ is odd, then
$\w_N^e \circ S$ has a $q$-expansion that is rational up to a
factor $\sqrt N$, so that Theorems~\ref {th:main} and~\ref {th:N-system}
are not applicable any more. Nevertheless, $\w_N^e$ may yield class
invariants; this is well-known for Weber's original functions in certain
cases.

\paragraph {Acknowledgements.}
The second author wants to thank the University
of Waterloo for its hospitality during his sabbatical leave; he also
wants to acknowledge the warm and studious atmosphere of the Asahi
Judo club of Kitchener where large parts of the results were proven.
The first author was partially funded by ERC Starting Grant ANTICS 278537.

\bibliographystyle {plain}
\bibliography {weber}

\end {document}